\documentclass[article]{amsart}
\usepackage[utf8]{inputenc}
\usepackage{lipsum}
\usepackage[all,cmtip]{xy}
\usepackage{tikz-cd}
\usepackage[english]{babel}
\usepackage{amsmath}
\usepackage{amsmath,amsthm}
\usepackage[euler]{textgreek}

\usepackage[a4paper,top=3.5cm,bottom=3cm,left=2.5cm,right=2.5cm,marginparwidth=2cm]{geometry}
\usepackage[nodisplayskipstretch]{setspace} 
\usepackage{amsthm}
\usepackage{comment}
\usepackage{tikz}
\usepackage{mathrsfs}
\usetikzlibrary{matrix}
\usepackage{tikz-cd}
\usepackage{amssymb}
\usepackage{dsfont}
\usepackage[all]{xy}
\usepackage{amsfonts}
\usepackage{listings}
\usepackage{enumerate} 
\usepackage{amsmath}
\usepackage{graphicx}
\usepackage[colorinlistoftodos]{todonotes}
\usepackage[colorlinks=true, allcolors=blue]{hyperref}
\author[A.\;GUTIERREZ TERRADILLOS]{ARMANDO GUTIERREZ TERRADILLOS}
\title[Short title]{A truncated Siegel-Weil formula and Borcherds forms}
\xdef\shorttitle{{A truncated Siegel-Weil formula and Borcherds products} -- \noexpand\MakeUppercase{\shortauthors}}
\newtheorem{thm}{Theorem}[section]
\newtheorem{defn}[thm]{Definition}

\newtheorem*{Thmintro}{Theorem}
\newtheorem{prop}[thm]{Proposition}

\newtheorem{lem}[thm]{Lemma}

\newtheorem{obs}[thm]{Remark}

\newtheorem{cor}[thm]{Corollary}

\newcommand{\operatork}{\omega(\alpha)}

\newcommand{\C}{\mathbb{C}}
\newcommand{\R}{\mathbb{R}}
\newcommand{\Q}{\mathbb{Q}}
\newcommand{\Z}{\mathbb{Z}}

\newcommand{\A}{\mathbb{A}}

\newcommand{\intfour}{\int_{-1/2}^{1/2}}
\newcommand{\RESE}{A_{-1}}
\newcommand{\CTE}{A_0}
\newcommand{\suma}{\sum_{j = 0}^1}
\newcommand{\SO}{SO}
\newcommand{\Sp}{\mathrm{Sp}}
\newcommand{\Hgrp}{\mathrm{H}}
\newcommand{\Kgrp}{\mathrm{K}}
\newcommand{\GSpin}{\mathrm{GSpin}}

\newcommand{\XK}{\mathrm{X}_{\mathrm{K}}}

\newcommand{\Proj}{\mathbb{P}}

\newcommand{\erf}{\mathrm{erf}}

\newcommand{\Gaussi}{\varphi^{\infty}_{z_0,\mu_j}}
\newcommand{\Gaussieva}{\varphi^{\infty}_{z,\mu_j}}

\newcommand{\gtau}{g_{\tau}}
\newcommand{\Gaussr}{\varphi^{\infty}_{z_0,\R}}

\newcommand{\Gausszero}{\varphi^{\infty}_{z_0,\mu_0}}
\newcommand{\Gaussone}{\varphi^{\infty}_{z_0,\mu_1}}

\newcommand{\Gauss}{\varphi^{\infty}_{z_0}}

\newcommand{\vol}{\mathrm{vol}}

\date{July 2022}

\begin{document}
\maketitle
\begin{abstract}
In this paper we use the regularized Siegel-Weil formula of Gan-Qiu-Takeda to obtain an expression of the integral of the theta function over the truncated modular curve. We apply this result to express the integral over the truncated modular curve of the logarithm of the Borcherds form
and we describe explicitly its asymptotic behaviour, and in particular the convergent and divergent contributions. The result provides a complement to the work of Kudla on integrals of Borcherds forms in a limiting case which falls out the range of applications.
\end{abstract}
\tableofcontents
\section{Introduction}
The integrals of the logarithm of the Borcherds forms have been related to zeta and $L-$values in a wide variety of papers \cite{kudla_2003_integral}, \cite{bruinier_2007_borcherds}, \cite{Gerard_twist}. Due to the geometric nature of the Borcherds forms, those integrals have also been used to understand the arithmetic degrees of arithmetic cycles of Shimura varieties, extending the knowledge of their Chow groups, \cite{ehlen_2017_cm}. In \cite{kudla_2003_integral} the author studies the integral of the logarithm of the Borcherds forms for certain quasiprojective Shimura varieties associated to the group $\mathrm{GSpin}$, obtaining an expression involving certain Fourier coefficients of Eisenstein series. Subsequently, Schofer's PhD thesis \cite{SchoeferThesis}, extended Kudla's results by computing the integral of the Borcherds form over all the CM-points of $\mathrm{GSpin}$ Shimura varieties. This computation enabled a new proof of the Gross-Zagier theorem on singular moduli \cite{SingularModuli}. One of the main tools in \cite{kudla_2003_integral} and \cite{SchoeferThesis} is the Siegel-Weil formula in the convergent range of Weil and the one proved in \cite{Kudla1988}. On account of the eventual divergence of the integral of the theta function over the modular curve, the integral of the logarithm of the Borcherds forms over the modular curve was not addressed in \cite{kudla_2003_integral}. An explicit example of the aforementioned integral is given by \[\Delta(\tau) = e^{2\pi i \tau}\prod_{n = 1}^{\infty}(1-e^{2\pi i \tau n})^{24},\] 
the so-called classical discriminant; according to  \cite[p.\;6]{bruinier_infinite} $\Delta(\tau)$ is the Borcherds form of the Jacobi theta function. In \cite[Corollary 5.4,\;p.\;21]{khn_1998_generalized} the author shows that the integral of the logarithm of the norm of $\Delta(\tau)$ diverges. This result points out the different nature between the integrals considered in \cite{kudla_2003_integral} and the one over the modular curve. A significant variation between \cite{kudla_2003_integral} and the present paper is that in order to understand our integral we have to replace the integral over the modular curve $X^{mod}$ by the integral over the truncated fundamental domain of the modular curve, denoted by $X^{mod,\hat{T}}$.

One of the main ingredients for our proof is the Siegel-Weil formula. The classical version of it, given in \cite{10.2307/1968644}, \cite{Siegel1951IndefiniteQF} and \cite{weil_1965_sur}, relates the integral of certain theta function with a special value of an Eisenstein series. Let $V$ be a rational quadratic space of dimension $m$ with Witt index $r$. For $n\geq 1$ the tuple $(\mathrm{Sp}_{n},O(V))$ forms a dual reductive pair. Given $\varphi\in \mathcal{S}(V^n(\A))$, where $\mathcal{S}$ denotes the space of Schwartz functions, we construct a theta function $\theta(g,h,\varphi):\;\mathrm{Sp}_{n}(\A)\times O(V)(\A)\to\C$. The convergence of $\int_{[O(V)]}\theta(g,h,\varphi)dh$ depends on the constants $m = \dim(V)$, $r$ and $n$. When $r = 0$ or $m-r>n+1$, we say that the datum is in the \textit{convergent range of Weil} and by \cite{weil_1965_sur} and \cite{kudla_1986_on}, the theta function is absolutely convergent and
\[k\int_{[O(V)]}\theta(g,h,\varphi)dh = E\left(g,\frac{m}{2}-1,\lambda(\varphi)\right) = \sum_{\gamma\in P^{\mathrm{Sp}_n}(\Q)\setminus \mathrm{Sp}_n(\Q)}\lambda(\varphi)(\gamma g),\]
where $k$ is an explicit constant, $P^{\mathrm{Sp}_n}$ is the Siegel parabolic of $\mathrm{Sp}_n$ and 
\[\lambda:\;S(V^n(\A))\to I_n\left(\frac{m}{2}-1,\chi_V\right),\]
is a map where $I_n\left(\frac{m}{2}-1,\chi_V\right)$ is the degenerated principal series representation of $\mathrm{Sp}_{n}(\A)$. It is known that $\lambda$ is a realization of the local theta correspondence of the identity, which has been extensively studied throughout \cite[\S\;2.\;p.\;17]{kudla_rallis}. In order to approach the remaining cases, under certain hypothesis on $V$ and $n$, \cite{kudla_rallis} and \cite{ichino_2001_on} developed a regularization of the theta function. It is based on a operator
\begin{equation}\operatork: \mathcal{S}(V^{n}(\A))\to \mathcal{S}(V^{n}(\A))_{abc},\end{equation}
where $\mathcal{S}(V^{n}(\A))_{abc} = \{\varphi\in\mathcal{S}(V^n(\A)),\;s.t.\; \int_{[O(V)]}\theta(g,h,\varphi)dh\textit{ absolutely convergent for all }g\}$. The map $\operatork$ is constructed using the action of an explicit element of the Hecke algebra of $O(V)$ on $\mathcal{S}(V^n(\A))$. This machinery allows us to define a meromorphic function $B(g,\varphi,s)$, \cite[\S\;3.5,\;p.\;18]{gan_2014_the}, which replaces the role of $\int_{[O(V)]}\theta(g,h,\varphi)dh$ in the classical Siegel-Weil formula. It is known that $B(g,\varphi,s)$ has a pole at $s = \frac{m-r-1}{2}$ of order at most $2$. We denote its $i-$th Laurent coefficient by $B_i(g,\varphi)$. The so-called first and second term identity of the Siegel-Weil formula relate $B_{-2}(g,\varphi)$ and $B_{-1}(g,\varphi)$ with special values of Eisenstein series (and their residues).\\

In this paper we consider the quadratic space given by the $2\times 2$ matrices with trace $0$, endowed with the determinant. It defines a quadratic space of dimension $3$ and Witt index $r = 1$. Geometrically it corresponds to the modular curve case. The first goal is to obtain an expression for the integral of the theta function associated to $\Gauss\in\mathcal{S}(V(\A))$; certain Schwartz function defined using the geometry of the modular curve. It turns out that this integral does not converge, hence in order to obtain some "truncated expression" for it we use the regularized Siegel-Weil formula. One of the available Theorems in this situation is given by \cite[Theorem\;8.1,\;(ii),\;p.\;35]{gan_2014_the}
\begin{equation}\label{intro1}\mathrm{CT}_{s = 1/2}\;E(g,s,\lambda(\Gauss)) = B_{-1}(g,\Gauss) +c\RESE (g,\lambda(\tilde{\varphi})),\end{equation}
where $\mathrm{CT}_{s = 1/2}\;E(g,s,\lambda(\Gauss))$ and $\RESE (g,\lambda(\tilde{\varphi}))$ are respectively the constant and residue terms of the Laurent series at $s = 1/2$ of these Eisenstein series and $c\in\C$. A drawback of this formula is that we can not recover information about $\int_{[O(V)]}\theta(g,h,\Gauss)dh$ directly. To that end we use the mixed model of the Weil representation  \cite[Proposition\;5.2.1,\;p.\;44 \& Proposition\;5.3.1,\;p.\;45]{kudla_rallis}. It allows us to factor the theta function as follows:
$$\theta(g,h,\Gauss) = \mathrm{Conv}(g,h,\Gauss)+\mathrm{Div}(g,h,\Gauss),$$
where $\int_{[O(V)]}\mathrm{Conv}(g,h,\Gauss)dh$ is absolutely convergent and $\int_{[O(V)]}\mathrm{Div}(g,h,\Gauss)dh$ diverges. Using the action of $\operatork$ we obtain a relation between $B_{-1}(g,\Gauss)$ and $\int_{[O(V)]}\mathrm{Conv}(g,h,\varphi)dh$.
\begin{Thmintro}[\ref{swcor}]
Let $\varphi^{\infty}_{z_0,\mu_j}$ be the Schwartz function defined in \eqref{schgauss} and let $\mathrm{Conv}(g_{\tau},h,\varphi^{\infty}_{z_0,\mu_j})$ be the absolutely convergent part of $\theta(g_{\tau},h,\varphi^{\infty}_{z_0,\mu_j})$. Then, there exists a $\tilde{\varphi}_j\in\mathcal{S}(V(\A))$ such that
\[\int_{[O(V)]}\mathrm{Conv}(g_{\tau},h,\varphi^{\infty}_{z_0,\mu_j})dh =  E(g_{\tau},1/2,\lambda(\varphi^{\infty}_{z_0,\mu_j}))+c\RESE (g_{\tau},\lambda(\tilde{\varphi}_{j})),\]
where $c\in\C$ does not depend on $g_{\tau}$.
\end{Thmintro}

Let us denote by $X^{mod}$ the modular curve without level and $\mathcal{F} := \{z = x+iy\in\mathcal{H},\;s.t.\;|x|\leq 1/2,\;|z|\geq 1\}$ its fundamental domain. Fix $\hat{T}>1$ and set $X^{mod,{\hat{T}}} := \{z = x+iy\in \mathcal{H}\;s.t.\;|x|\leq 1/2,\;|z|\geq 1,\;y<\hat{T}\}$. The main body of this paper is devoted to obtain an explicit expression of
\[\int_{X^{mod,\hat{T}}}\log||\Psi(f)(z)||_{Pet}d\mu(z),\]
where $f(\tau) = \sum_{\substack{n\in\Z\\j \in\{0,1\}}}c_{\mu_j}(n)q^n\varphi_{\mu_j}\in M_{-1/2,\Z^3}^!$, $\Psi(f)(z)$ is the Borcherds form of $f$, $||\cdot||_{Pet}$ is the Peterson norm and $d\mu$ is the hyperbolic measure. The function $\Psi(f)$ is closely related to the singular theta lift, hence the computation can be reduced to  \[\int_{X^{mod,\hat{T}}}\int_{X^{mod}}^{\bullet}\theta^{Sieg}(\tau,z)f(\tau)d\mu(\tau)d\mu(z),\]
where $\int^{\bullet}_{X^{mod}}$ is the regularization proposed in \cite{borcherds_1998_automorphic} to ensure the convergence of the singular theta lift. Due to the behaviour of the Fourier constant term of $\theta^{Sieg}(\tau,z)$ the order of the previous two integrals can not be exchanged. With the aim of accomplishing the computation we treat separately the integrals involving the non constant and constant terms of $\theta^{Sieg}(\tau,z)$. The first one is approached following the method developed in \cite{kudla_2003_integral} that requires the truncated version of the Siegel-Weil formula stated above. The second integral is computed via an unfolding of the theta function with the integral. To ensure the convergence in this unfolding we introduce an auxiliary Eisenstein series and to conclude we apply the truncated version of the Rankin-Selberg formula developed by Zagier.
\begin{Thmintro}[\ref{mainresult} \& \ref{MainCorollary}]
Let $f\in M_{-1/2,\Z^3}^!$ be a vector valued modular form. We compute an explicit expression for \[\int_{X^{mod,\hat{T}}}\log||\Psi(f)(z)||_{Pet}d\mu(z),\] where the constant term is expressed in terms of Fourier coefficient of an Eisenstein series, values of the Riemann zeta function and special functions. When $c_{\mu_0}(0)\neq 0$, the integral diverges as $\mathrm{log}(\hat{T})$. In the case where $c_{\mu_0}(0) = 0$, the integral converges and we have 
\[\int_{X^{mod}}\log||\Psi(f)(z)||_{Pet}d\mu(z) = -\frac{\vol(X^{mod})}{2}\suma\sum_{m\in j/2 +\Z^{> 0}} c_{\mu_j}(-m)\kappa_{\mu_j}(m),\]
where $\kappa_{\mu_j}(m) := \lim_{v\to\infty}b(m,v,\varphi_{\mu_j})$ is the limit of the first term in the Laurent expansion of the $m$-th Fourier coefficient of the unramified Eisenstein series of weight $3/2$.
\end{Thmintro}
The paper is organized as follows. Section 2 presents the geometric setting, establishing the general framework. Section 3 discusses the connection between the regularized theta correspondence and the adelic theta correspondence. In Section 4, we prove a truncated version of the Siegel-Weil formula. Section 5 forms the core of the paper, where we establish the main result. The proof is divided into two cases: the ordinary case, which relies on the truncated Siegel-Weil formula, and the limit case, which involves unfolding the integral with the theta series. The final section addresses several technical computations needed in Section 5.
\subsubsection{Notation}\label{notation}
Given $G$ an algebraic group defined over $\Q$, we denote by $G_{\Q_p} := G_{\Q}\times_{\mathrm{Spec}(\Q)}\mathrm{Spec}(\Q_p)$ the base change of $G$ to $\Q_p$. We use the notation $[G]= G(\Q)\setminus G(\A)$. Given $(V,q)$ a quadratic space we denote by $m := \mathrm{dim}\;V$, $r$ its Witt index, $V_{an}$ a maximal anisotropic subspace and $(\cdot,\cdot)$ the bilinear form associated to $q$. Let us denote by $O(V)$ the orthogonal algebraic group. Furthermore, given a rational symplectic space of dimension $2n$ we denote by $\mathrm{Sp}(W)$ its symplectic group, which is an algebraic group defined over $\Q$. There is a choice of basis for $W$ so that $\mathrm{Sp}(W)\simeq \mathrm{Sp}_n$. Given a topological space $T$ and a subspace $S$ we denote by $\mathrm{char}_S(x):\;T\to \{0,1\}$ the characteristic function of $S$. We fix $\psi_{\infty}$ the character on $\R$ given by $\psi_{\infty}(x) = e^{2\pi i x}$ and the unique characters $\psi_p$ on $\Q_p$ that are additive, whose restriction to $\Z_p$ is trivial and they satisfy that $\psi_p(p^{-1}) = e^{\frac{-2\pi i }{p}}$.\\

Let $\mathcal{H}$ be the Poincaré half plane. We fix two complex variables $\tau = u+iv$ and $z = x+iy$. Given a meromorphic function $F(s)$, $s_0\in\C$ and $C_{s_0}$ a closed curve around $s_0$ we denote by
\[\mathrm{FT}_{s = s_0}F(s) = \int_{C_{s_0}}\frac{F(s)}{(s-s_0)^2}ds,\] the first term in the Laurent series of $F(s)$ at $s = s_0$, by 
\[\mathrm{CT}_{s = s_0}F(s) = \int_{C_{s_0}}\frac{F(s)}{(s-s_0)}ds,\]
the constant term in the Laurent series at $s = s_0$ and by \[\mathrm{Res}_{s = s_0}F(s) = \int_{C_{s_0}}F(s)ds,\] the residue of $F(s)$ at $s = s_0$. The Euler-Mascheroni constant is denoted by $\gamma := -\Gamma'(1)$.\\

This relevant constant will appear in the computations:
\[A := \mathrm{erf}\left(\sqrt{\pi/2}\right)(\log(2)+2\gamma)+\frac{\pi^{1/2}}{2}(\Gamma'(-1/2,\pi/2)-\Gamma(-1/2,\pi/2)\log(\pi/2)),\]
where $\Gamma(a,b)$ is the incomplete Gamma function.
\subsubsection{Acknowledgements}
I would like to thank Luis García for his advice on the subjects of this paper. The paper is part of my PhD thesis under the supervision of Gerard Freixas and Victor Rotger. Finally, I would like to thank the anonymous referee for the valuable corrections and comments that helped me to improve this article. This project has received funding from the European Research Council (ERC) under the European Union’s
Horizon 2020 research and innovation programme (grant agreement No 682152).
\section{Borcherds forms and the Siegel-Weil formula}
\subsection{GSpin Shimura varieties}
Throughout this section we fix $\left(V^{(p,2)},Q\right)$ a rational quadratic space of dimension $m$ and signature $(p^+,p^-) = (p,2)$ with $p\geq 1$. We let $(\cdot,\cdot)$ be the symmetric bilinear form associated to $Q$. The general Spin group $\widetilde{\Hgrp} := \GSpin _{V^{(p,2)}}$ is defined as the central extension of the special orthogonal group $SO(V^{(p,2)})$ that fits in the following exact sequence:
\begin{equation}\label{eq51}1\to \mathbb{G}_m\to \widetilde{\Hgrp}\to SO(V^{(p,2)})\to 1.\end{equation}
Using the quadratic form $q$ we fix an orthogonal decomposition \begin{equation}\label{factorspace}V^{(p,2)}(\R) = V^{(p,2),+}(\R)\oplus V^{(p,2),-}(\R),\end{equation} where $V^{(p,2),+}(\R)$ is positive definite and $V^{(p,2),-}(\R)$ is negative definite. The group $S(O(2)\oplus O(p))$ is a maximal compact subgroup of $\SO(V^{(p,2)})(\R)$. Following \cite[Proposition\;1.6,\;p.\;12]{Milne_Shim} we construct a Hermitian symmetric domain as follows:
\[D_{V^{(p,2)}} = \SO(V^{(p,2)})(\R)/\left(S(O(2)\oplus O(p))\right).\]
Let us fix $\Kgrp<\widetilde{\Hgrp}(\A_f)$ a compact open subgroup. We define the Shimura variety associated to $\widetilde{\Hgrp}$ with level $\Kgrp$ by the following double quotient:
\begin{equation}\label{eq54} \XK = \widetilde{\Hgrp}(\Q)\setminus D_{V^{(p,2)}}\times \widetilde{\Hgrp}(\A_f)/\Kgrp.\end{equation}
It is a scheme defined over a number field called \textit{reflex field} \cite[def.\;11.1,\;p.\;107]{Milne_Shim}. According to \cite[Proposition\;5.13,\;p.\;57]{Milne_Shim} this double quotient is isomorphic to 
\begin{equation}\label{eq99}
\bigsqcup_{g\in\mathcal{C}}\Gamma_{g}\setminus D_{V^{(p,2)}},\end{equation}
where $\mathcal{C}$ is a set of representatives of $\widetilde{\Hgrp}(\Q)_+\setminus \widetilde{\Hgrp}(\A_f)/\Kgrp$ and $\Gamma_g = g\Kgrp g^{-1}\cap \widetilde{\Hgrp}(\Q)_+$ with $\widetilde{\Hgrp}(\Q)_+ = \widetilde{\Hgrp}(\Q)\cap \widetilde{\Hgrp}(\R)^+$ the intersection of the rational points of $\widetilde{\Hgrp}$ and the identity component of the real points of $\widetilde{\Hgrp}$.\\

Let us consider an integral lattice $M$ of $V^{(p,2)}$ so that $M\otimes_{\Z}\Q = V^{(p,2)}$. We denote the set of isometries of $M$ by $\SO(M)$. The dual lattice of $M$ is defined as follows:
\[M' := \{l\in V^{(p,2)}(\R),\;s.t.\;(\lambda,l)\in\Z,\;for\;all\;\lambda\in M\}.\]
 Moreover set
\[\Gamma_M := \{g\in \SO(M),\;s.t.\;g\;acts\;trivially\;on\;M'/M\}.\]
The quotient
\[\mathrm{X}_{\Gamma_M} := \Gamma_M\setminus D_{V^{(p,2)}},\]
admits a projective compactification $\mathrm{X}_{\Gamma_M}^*$ by the theory of Baily-Borel \cite{lipnowski_bailyborel}.
\begin{prop}\label{geom20}
There exists an isomorphism of $\mathcal{C}^{\infty}-$manifolds
$$D_{V^{(p,2)}} \simeq \mathcal{K} := \{Z\in \Proj V^{(p,2)}(\C)\;s.t.\;(Z,Z) = 0,\;(Z,\overline{Z})>0\},$$
where $\mathcal{K}$ is called the \textit{projective cone model}.
\end{prop}
\begin{proof}
See \cite[p.\;76]{bruinier.Bor}.
\end{proof}
Using \cite[p.\;77]{bruinier.Bor} one can show that there exists an isomorphism of $\mathcal{C}^{\infty}-$manifolds 
\begin{equation}\label{isomquad}\omega:\;D_{V^{(p,2)}}\to \mathrm{Gr}(V^{(p,2)}),\end{equation}
where $\mathrm{Gr}(V^{(p,2)})$ is the connected component of the Grassmanian of negative definite subspaces of $V^{(p,2)}$. For $z\in D$ and $x\in V^{(p,2)}(\R)$, let $x_{z} := x_{\omega(z)}$ denote the projection of $x$ onto the subspace $\omega(z)\in \mathrm{Gr}(V^{(p,2)})$. Similarly, we define $x_{z^{\perp}} := x_{\omega(z)^{\perp}}$, the projection of $x$ onto the orthogonal complement $\omega(z)^{\perp}\in\mathrm{Gr}(V^{(p,2)})$.
\subsubsection{The modular curve without level}\label{modexample}
This paper focuses on studying the modular curve without level structure, which can be described using the framework outlined above. Let $(V,q)$ be the rational quadratic space defined as:
\[V:= \left\{v(x_1,x_2,x_3) := \left(\begin{smallmatrix}
    x_1&x_2\\x_3&-x_1
\end{smallmatrix}\right)\in M_2(\Q)\right\},\]
with inner product given by $(X,Y) = -\mathrm{Tr}(XY)$. Furthermore, we denote the quadratic form associated with this inner product by $q$. This quadratic space has $\mathrm{det}(V) = 1$ and the real quadratic space $V(\R) := V\otimes_{\Q}\R$ has signature $(1,2)$. We fix a basis $\{f_1,f_2,f_3\}$ where $f_1 = v(0,1,0),\;f_2 = v(1,0,0),\;f_3 = v(0,0,1)$.\\
Next, we consider the lattice $L := V\cap M_2(\Z)\simeq  \Z\oplus\Z\oplus\Z\subset V$, with the isomorphism taken with respect to the previously defined basis. In this context, the quotient space $\mathrm{X}_{\Gamma_L}$ satisfies the following isomorphism of differentiable manifolds:
\[X_{\Gamma_L} = \Gamma_L\setminus D_V\simeq X^{mod} := SL_2(\Z)\setminus \mathcal{H}.\]
Moreover, the lattice $L$ satisfies 
\[L'/L \simeq \left(\Z\oplus\frac{1}{2}\Z\oplus\Z\right)/\left(\Z\oplus\Z\oplus\Z\right) \simeq \frac{1}{2}\Z/\Z.\]

\subsection{Borcherds forms and the theta correspondence}
Borcherds forms are meromorphic modular forms constructed from vector-valued weakly holomorphic modular forms through techniques in complex geometry and analysis. Since a key tool in this paper is the Siegel-Weil formula from \cite{gan_2014_the}, an important result in automorphic representations theory, it is essential to clarify the connection between Borcherds forms and the global theta correspondence. In this section, we will define Borcherds forms on modular curves and explain their relationship to the framework of automorphic representations, as described in \cite{kudla_2003_integral}.

Throughout this section, let $V$ denote the rational quadratic space defined in \S\ref{modexample} and let $L = V\cap M_2(\Z)$. Consequently, we have the isomorphism $D_V\simeq\mathcal{H}$. For convinience, we will denote the orthogonal group of $V$ by $H := O(V)$.
\subsubsection{Borcherds forms}\label{BorcherdsFormSubsec}
We define $\varphi_{\mu_0} := \mathrm{char}_{(\Z\oplus\Z\oplus\Z)\otimes \A_f}$ and $\varphi_{\mu_1} := \mathrm{char}_{(\frac{1}{2}\Z\oplus\Z\oplus\Z)\otimes \A_f}$. We then set $S_L := \{\varphi_{\mu_0},\varphi_{\mu_1}\}$.
\begin{lem}\label{121}
The Schwartz functions $\varphi_{\mu_0}$ and $\varphi_{\mu_1}$ are $\prod_{p\nmid \infty}H(\Z_p)$-invariant.
\end{lem}
\begin{proof}
The group $H(\A_f)$ acts on $\mathcal{S}(V(\A_f))$ by left translation and the group $\prod_{p\nmid \infty}H(\Z_p)$ preserves the lattices $\hat{\Z}_f^3$ and $\hat{\Z}_f\oplus \frac{1}{2}\hat{\Z}_f\oplus \hat{\Z}_f$.
\end{proof}
\begin{defn}\label{siegeltheta}
\em
The \textit{Siegel theta function} is defined by
\begin{align*}
    \Theta_L^{Sieg}:\;\mathcal{H}\times D_V&\to\C[S_L],\\
    (\tau,z)&\mapsto \sum_{j = 0}^1\theta^{Sieg}_{\mu}(\tau,z)\varphi_{\mu},
\end{align*}
where $\theta^{Sieg}_{\mu_0}(\tau,z) := \sum_{\lambda\in L}e^{2\pi i(q(\lambda_z)\tau+q(\lambda_{z^{\perp}})\overline{\tau})}$ and $\theta^{Sieg}_{\mu_1}(\tau,z) := \sum_{\lambda\in \frac{1}{2}f_2+L}e^{2\pi i(q(\lambda_z)\tau+q(\lambda_{z^{\perp}})\overline{\tau})}$
\em
\end{defn}
\begin{defn}
\em
The \textit{real metaplectic group} $\mathrm{Mp}_2(\R)$ is the following double cover of $\mathrm{SL}_2(\R)$: the elements are given as pairs $(M,\phi(\tau))$ where $M = \left(\begin{smallmatrix}a&b\\c&d\end{smallmatrix}\right)\in \mathrm{SL}_2(\R)$ and $\phi(\tau)$ is a holomorphic square root of $c\tau+d \in\C$. Given $(M_1,\phi_1(\tau)),\;(M_2,\phi_2(\tau))\in \mathrm{Mp}_2(\R)$ the group law in $\mathrm{Mp}_2(\R)$ is defined by $(M_1,\phi_1(\tau))\cdot(M_2,\phi_2(\tau)) = (M_1M_2,\phi_1(M_2\tau)\phi_2(\tau))$.
\em
\end{defn}
We denote by
\begin{equation}\label{Weilrara}\rho_L:\;\mathrm{Mp}_2(\Z)\to\mathrm{Aut}(S_L),\end{equation}
the Weil representation associated to $L$ defined in \cite[p.\;15]{bruinier.Bor}. 
\begin{obs}
The Siegel theta function is a real analytic function in $(\tau,z)\in\mathcal{H}\times D_V$ and satisfies the following transformation law:
\[\Theta_L^{Sieg}(M\tau,\gamma z) = \phi(\tau)^2\overline{\phi(\tau)}^p\rho_L(M,\phi)\Theta_L^{Sieg}(\tau,z),\]
where $(M,\phi(\tau))\in \mathrm{Mp}_2(\Z)$ and $\gamma\in \Gamma_L$.
See \cite[Theorem\;2.1,\;p.\;40]{bruinier.Bor} for more details.
\end{obs}
\begin{defn}
\em
Let $k\in \frac{1}{2}\Z$. A \textit{vector valued weakly holomorphic modular form} of weight $k$ is a smooth function $f:\mathcal{H}\to S_L$ that is holomorphic on $\mathcal{H}$, meromorphic at the cusp $\infty$ and for any $(M,\phi(\tau))\in \mathrm{Mp}_2(\Z)$ it satisfies the following transformation law:
\[f(\tau) = \phi(\tau)^{-2k}\rho^*_L(M,\phi(\tau))^{-1}f(M\tau),\]
where $\rho^*_L(M,\phi(\tau))$ is the dual of $\rho_L$. The space of vector valued weakly holomorphic modular forms of weight $k$ will be denoted by $M^{!}_{k,L}$.
\em
\end{defn}
Let us denote by $\left<\cdot,\cdot\right>$ the inner product on $S_L$ given by
\[\left<\sum_{\mu = 0}^1a_{\mu}\varphi_{\mu},\sum_{\nu = 0}^1b_{\nu}\varphi_{\nu}\right> := \sum_{\gamma = 0}^1a_{\gamma}\overline{b}_{\gamma}.\]
\begin{defn}\label{4}
\em
Let $k\in \frac{1}{2}\Z$ and $f\in M^{!}_{k,L}$. The \textit{singular theta lift} of $f$ is defined as follows:
\[\Phi(f)(z) := \mathrm{CT}_{\sigma = 0}\lim_{T\to\infty}\int_{\mathcal{F}^T}\left<f(\tau),\Theta^{Sieg}_L(\tau,z)\right>v^{1-\sigma}d\mu(\tau),\]
where $\mathcal{F}^T := \{\tau = u+iv\in\mathcal{H},\;s.t.\;|u|\leq 1/2,\;|\tau|\geq 1\;v<T\}$.
\em
\end{defn}
To lighten the notation we will use 
\[\int^{\bullet}\left<f(\tau),\Theta^{Sieg}_L(\tau,z)\right>d\mu(\tau) := \mathrm{CT}_{\sigma = 0}\lim_{T\to\infty}\int_{\mathcal{F}^T}\left<f(\tau),\Theta^{Sieg}_L(\tau,z)\right>v^{1-\sigma}d\mu(\tau).\]

\begin{thm}\label{mainBorcherds}
Let $k\in \frac{1}{2}\Z$ and $f\in M^!_{k,L}$ so that $f(\tau) = \sum_{j = 0}^1\sum_{n\gg-\infty}c_{\mu_j}(n)q^n\varphi_{\mu_j}$ with $c_{\mu_j}(n)\in \Z$ for $n\leq 0$. There exists a function 
\[\Psi(f):\;D_V\to\C,\]
such that
\begin{enumerate}
    \item $\Psi(f)$ is a meromorphic modular form of weight $c_{\mu_0}(0)/2$ with respect to the group $\Gamma_L$.
    \item The function $\Psi(f)$ satisfies the following relation:\[\log|\Psi(f)(z)| = -\frac{\Phi(f)(z)}{4}- \frac{c_{\mu_0}(0)}{2}\left(\log|y|-\gamma/2+\log\sqrt{2\pi}\right),\]
    where $z = x+iy$.
\end{enumerate}
The function $\Psi(f)$ is called the Borcherds form of $f$.
\end{thm}
\begin{proof}
This is \cite[Theorem\;13.3,\;p.\;48]{borcherds_1998_automorphic}.
\end{proof}
\subsubsection{Automorphic theta kernel and its relation with the singular theta lift}
Let $G = \mathrm{Mp}_2$ be the metaplectic group as defined in \cite[\S1, p. 10]{sweet_1990_the}. The Borcherds forms introduced in Theorem \ref{mainBorcherds} are connected to the theta correspondence for the dual reductive pair $(G,H) := (\mathrm{Mp}_2,O(V))$. In this subsection, we will define the global theta function for this dual reductive pair and elucidate its relationship with the Borcherds forms.\\

The vector spaces $\mathcal{S}(V(\Q_p))$, $\mathcal{S}(V(\A_f))$ and $\mathcal{S}(V(\A))$ denote the Schwartz spaces of $V(\Q_p)$, $V(\A_f)$ and $V(\A)$ respectively. We denote by $\omega = \otimes_{p}\omega_p$ the global Weil representation for the dual reductive pair $(G,H)$, we refer the reader to \cite[\S\;1.2,\;p.\;16]{sweet_1990_the} for a detailed discussion.
\begin{defn}
\em
The global theta function on $G(\A)\times H(\A)\times \mathcal{S}(V(\A))$ is defined as follows:
\[
   \theta(g,h,\varphi) = \sum_{x\in V(\Q)}\omega(g,h)\varphi(x).
\]
\em
\end{defn}

Given $\tau = u+iv\in\mathcal{H}$ we set
$g_{\tau} := \left(\begin{smallmatrix}1&u\\ &1\end{smallmatrix}\right)\left(\begin{smallmatrix}v^{1/2}& \\ &v^{-1/2}\end{smallmatrix}\right)\in \mathrm{SL}_2(\R)$. Since there is no danger of confusion, we will also denote by $g_{\tau} := (g_{\tau},1)\in \mathrm{Mp}_2(\R)$. 
Let $R(x,z) := -(x_z,x_z) = |(x,\omega(z))|^2|y|^{-2}$ where $\omega$ is the map defined in \eqref{isomquad}. We use the following notation:
\[(x,x)_z := (x,x)+2R(x,z).\]
\begin{defn}\label{GaussianDef}
    \em
    The \textit{Gaussian} is defined as the function $\varphi^{\infty}_{z_0,\R}(x) := e^{-\pi(x,x)_{z_0}}  \in\mathcal{S}(V(\R))$ with $z_0\in \mathrm{Gr}(V)$ corresponding to $i\in \mathcal{H}$ under the isomorphisms $\omega^{-1}:\mathrm{Gr}(V)\to D_V$ and $D_V\simeq \mathcal{H}$.
    \em
\end{defn}
\begin{lem}\label{weightGaussian}
    The Gaussian satisfies that $\omega(k_{\theta})\varphi^{\infty}_{z_0,\R}(x) = e^{-\frac{i\theta}{2}}\varphi^{\infty}_{z_0,\R}(x)$ for $k_{\theta} = \left(\begin{smallmatrix}cos(\theta)&sin(\theta)\\-sin(\theta)&cos(\theta)\end{smallmatrix}\right) \in \mathrm{SL}_2(\R)$. Furthermore, the action of $S(O(2)\oplus O(1))$ fixes $\varphi^{\infty}_{z_0,\R}(x)$.
\end{lem}
\begin{proof}
    The first assertion follows from \cite[(1.28), p. 10]{kudla_2003_integral}. The second assertion follows by the definition of $D_V$.
\end{proof}
\begin{obs}
Using the discussion of \cite[\S\;1,\;p.\;6]{kudla_2003_integral} and \cite[Lemma\;1.1,\;p.\;11]{kudla_2003_integral} one can show that the singular theta lift $\Phi(f)$ defined in \ref{4} may be expressed using the adelic theta function. In fact the following relation holds:
\[\Phi(f)(z) = \mathrm{CT}_{\sigma = 0}\lim_{T\to\infty}\int_{\mathcal{F}^T}\sum_{j = 0}^1f_{\varphi}(\tau)\theta(g_{\tau},h_{z},\varphi^{\infty}_{z_0,\R}(x)\otimes\varphi_{\mu_j}) v^{\frac{p}{2}-1-\sigma}d\mu(\tau),\]
where $h_{z}\in SO(V)(\R)$ satisfies $h_{z}\cdot z_0 = z$.
\end{obs}
\subsection{Regularized Siegel-Weil formula}
Let $V$ denote the rational quadratic space defined in \S\ref{modexample}. In this section, we will discuss the regularized Siegel-Weil formula for the dual reductive pair $(G,H) := (\mathrm{Mp}_2,O(V))$, which establishes a connection between the integral of the theta function over $[H]$ and the Eisenstien series of $G$. We will begin by introducing the relevant Eisenstein series defined over $G(\A)$. Following this, we will address the regularization of the integral of the theta function and present the Siegel-Weil formula from \cite{gan_2014_the}, which involves two terms of the Laurent expansion of the Eisenstein series. Finally, we will derive and analyze specific complex variable Eisenstein series that will play a central role in this paper.
\subsubsection{Eisenstein series}
Let us identify $\mathrm{Mp}_2(\A)\simeq \mathrm{SL}_2(\A)\times\{\pm1\}$, where the multiplication on the right hand side is given by $(g_1,\xi_1)\cdot(g_2,\xi_2) = (g_1g_2,\xi_1\xi_2\beta(g_1,g_2))$, with $\beta(g_1,g_2)$ the cocycle defined in \cite[p. 37]{sweet_1990_the}. Using this description we define the following two subgroups:
\[\tilde{T}(\A) = \left\{[m(a),\xi],\;m(a) := \left(\begin{smallmatrix}a& \\ & a^{-1} \end{smallmatrix}\right),\;s.t.\;a\in \A^{\times},\;\xi\in\{\pm1\}\right\},\]
\[\tilde{U}(\A) = \left\{[n(b),1],\;n(b) := \left(\begin{smallmatrix}1&b \\ & 1\end{smallmatrix}\right),\;s.t.\;b\in\A\right\}.\]
Furthermore we will denote $\tilde{B}(\A) := \tilde{T}(\A)\tilde{U}(\A)$, and let $B(\A)$ denote the adelic points of the upper triangular parabolic subgroup of $\mathrm{SL}_2(\A)$. Let $\chi_V:\Q^{\times}\setminus\A^{\times}\to\C^{\times}$ be the adelic character associated to the quadratic space defined in \S\ref{modexample}:
\[\chi_V(x) := \prod_p\left(x_p,-\det(V)\right)_{p} = \prod_p\left(x_p,-1\right)_{p},\]
where $x = (x_p)_p\in \A^{\times}$. It determines the genuine character $\tilde{\chi}_V([m(a),\xi]) := \xi\chi_{\psi}(a)\chi_{V}(a)$ of $\tilde{T}(\Q)$, where $\chi_{\psi}$ is the Weil index, a character defined in \cite[Appendix\;p.\;365]{rao}. 
\begin{defn}
\em
The \textit{degenerate principal series representation} of $G(\A)$ is given by
\[I_1(s,\chi_V) := \mathrm{Ind}_{\tilde{B}(\A)}^{G(\A)}\left(\tilde{\chi}_V|\cdot|^{s}\right),\]
where the induction is normalized.
We consider the elements of the induction that are smooth functions $\Phi(g,s)$ on $G(\A)$. 
\em
\end{defn}
\begin{defn}\label{Eisdef}
\em
Let $\Phi(g,s)\in I_1(s,\tilde{\chi}_V)$ be a holomorphic section. The \textit{Eisenstein series} of $G(\A)$ associated to $\Phi(g,s)$ is defined by
\[E(g,s,\Phi) := \sum_{\gamma\in B(\Q)\setminus \mathrm{SL}_2(\Q)}\Phi(\gamma g,s),\]
\em
\end{defn}
\begin{obs}
The above Eisenstein series are absolutely convergent in the half plane $\mathrm{Re}(s)>1$ \cite[Theorem\;7.1,\;p.\;34]{arthur_2005_an}. Furthermore they have meromorphic continuation in the variable $s$ \cite[Theorem\;7.2,\;p.\;35]{arthur_2005_an} and, by \cite[Proposition\;6.1\;p.\;28]{gan_2014_the}, they have a pole of order at most $1$ at $s = \frac{1}{2}$.
\end{obs}
\begin{obs}\label{intermap}
Let $p$ be a place of $\Q$. The local theta correspondence of the identity defines the following $G(\Q_p)-$intertwining  maps
\begin{align}\label{intermap2}\lambda_p:\;\mathcal{S}(V(\Q_p))&\to I_1\left(1/2,\tilde{\chi}_V\right),\\
\varphi\mapsto& \omega_p(g)\varphi(0),\nonumber
\end{align}
For a further discussion of this, we refer the reader to \cite[III.\;5,\;p.\;50]{kudla_1986_on}. The map $\lambda := \otimes_{p}\lambda_p$ realizes the global theta correspondence of the identity and relates the integral of the theta function with the Eisenstein series via the so-called Siegel-Weil formula. 
\end{obs}
\subsubsection{The regularized Siegel-Weil formula}\label{theregularizedsection}
In this subsection, we will begin by describing the regularized theta integral. This is defined by applying an operator to the Schwartz function, ensuring that the theta function becomes rapidly decreasing in the $H(\A)$-coordinate. Additionally, we need to include an auxiliary Eisenstein series in the definition of the regularized theta integral, which introduces the complex variable.\\

Let $p$ be a prime satisfying the conditions of \cite[p. 209]{ichino_2001_on}. The spherical Hecke algebras $\mathcal{H}_p^{G}$ and $\mathcal{H}_p^{H}$ of $G(\Q_p)$ and $H(\Q_p)$ respectively act on $\mathcal{S}(V(\A))$ by the Weil representation:
\begin{equation}\label{HeckeAlgWeil}\omega(f_G)\varphi(x) := \int_{G(\Q_p)}f_G(g_p)\omega(g_p)\varphi(x)dg_p,\;\;\;\omega(f_H)\varphi(x) := \int_{H(\Q_p)}f_H(h_p)\omega(h_p)\varphi(x)dh_p,\end{equation}
where $f_G\in \mathcal{H}_p^G$ and $f_H\in \mathcal{H}_p^H$. In \cite[Lemma 1.3, p. 208]{ichino_2001_on}, the author defines an element $\alpha: = \alpha_{1,0,\eta}\in\mathcal{H}_p^{G}$. According to \cite[Proposition\;1.5,\;p.\;209]{ichino_2001_on}, it has the property that \begin{equation}\label{Mainintegral}\int_{[H]}\theta(g,h,\omega(\alpha)\varphi)dh,\end{equation}
is absolutely convergent for all $g\in G(\A)$. The map $\theta$ given in \cite[Proposition 1.11, p. 206]{ichino_2001_on}, allows us to define an element $\alpha_H := \theta(\alpha)\in\mathcal{H}_p^H$, so that 
\begin{equation}\label{equalityoperators}\omega(\alpha) = \omega(\alpha_H).\end{equation}
To connect the integral \eqref{Mainintegral} with an Eisenstein series, the authors introduced an \textit{auxiliary Eisenstein series} in \cite[p. 46]{kudla_rallis}. The rest of this subsection will be devoted to define the auxiliary Eisenstein series and to state the second term identity of the regularized Siegel-Weil formula given in \cite{gan_2014_the}.\\ 

We recall that the Levi decomposition of the upper triangular parabolic of $H$, denoted here by $Q$, is
$$Q = M_QN_Q,$$
where the $\Q-$points of the Levi subgroup $M_Q$ are described as follows:
$$M_Q(\Q) = \left\{ m(a,h_0) = \left(\begin{smallmatrix} a & & \\ &h_0& \\ & &a^{-1}\end{smallmatrix}\right)\;s.t.\;a\in \Q^{\times},\;h_0\in \{\pm1\}\right\},$$
and the $\Q-$points of the unipotent subgroup $N_Q$ are
$$N_Q(\Q) = \left\{ n(c) = \left(\begin{smallmatrix} 1 &c &-\frac{1}{2}(c,c) \\ &1&-c^t \\ & &1\end{smallmatrix}\right)\;s.t.\;c\in \Q\right\}.$$
The Iwasawa decomposition provides the following equality:
$$H(\A) = Q(\A)K_{H},$$
where $K_{H} = (O(2)\oplus O(1))\times\prod_{p\;finite}H(\Z_p)$ is the maximal compact subgroup of $H(\A)$. Combining the previous two decompositions we can factor every $h\in H(\A)$ by
$$h = n(c)m(a,h_0)k.$$
To lighten notation we will use $|a(h)| := |a|_{\A}$. Using the previous datum we define the function
\begin{equation}\label{definitionPsi}\Psi(h,s) := |a(h)|^{s+1/2},\end{equation}
where $s\in\C$. 
\begin{defn}\label{105}
\em
The \textit{auxiliary Eisenstein series} is defined by
\begin{equation*}
E(h,s) = \sum_{\gamma\in Q(\Q)\setminus H(\Q)}\Psi(\gamma h,s).
\end{equation*}
\em
\end{defn}
\begin{obs}
This series is absolutely convergent when $Re(s)>1/2$ and moreover $E(h,s)$ has meromorphic analytic continuation to $\C$, see \cite[p.\;47]{kudla_rallis}.
\end{obs}
\begin{prop}\label{124}
The function $E(h,s)$ has a simple pole at $s = 1/2$ with constant residue $2$.
\end{prop}
\begin{proof}
It is \cite[Proposition\;5.4.1,\;p.\;48]{kudla_rallis}.
\end{proof}
The Hecke algebra $\mathcal{H}_p^H$ of $H(\Q_p)$ acts on an automorphic form $\varphi$ defined over $O(V)(\A)$ by convolution:
\begin{equation}\label{ActionHeckeAut}f_{H}* \varphi(h) = \int_{H(\Q_p)}f_H(h_p)\varphi(hh_v^{-1})dh_v.\end{equation}
Since the Hecke algebra $\mathcal{H}_p^{H}$ is commutative, it acts on irreducible representations by multiplication by a scalar. In particular, according to \cite[p. 15]{gan_2014_the}, the element $\alpha_H$ satisfies 
\begin{equation}\label{eigenEisenstein}
    \alpha_H* E(h,s) = P_{1,1}(s)E(h,s),
\end{equation}
where $P_{1,1}(s):\C\to\C$ is the function so that $\alpha_H* \Psi(h,s) = P_{1,1}(s)\Psi(h,s)$, where $\Psi(h,s)$ is defined as in \eqref{definitionPsi}. As established in \cite[Lemma 3.8 (ii), p. 18]{gan_2014_the} this function satisfies that $P_{1,1}(1/2) = 0$.\\

We define
$$\mathcal{E}(s,g,\varphi) := \frac{1}{2P_{1,1}(s)}\int_{[H]}\theta(g,h,\operatork\varphi)E(h,s)dh.$$
By Proposition \ref{124} and \cite[Lemma 3.8 (ii), p. 18]{gan_2014_the}, the function $\mathcal{E}(s,g,\varphi)$ has a pole at $s = 1/2$ of order at most $2$. This behaviour arises from the pole of the auxiliary Eisenstein series and the vanishing of $P_{1,1}(1/2)$. We denote the Laurent expansion of $\mathcal{E}(s,g,\varphi)$ at $s = 1/2$ by
$$\mathcal{E}(s,g,\varphi) = \frac{B_{-2}(g,\varphi)}{(s-1/2)^2}+\frac{B_{-1}(g,\varphi)}{(s-1/2)}+B_0(g,\varphi)+O(s-1/2).$$
\begin{thm}\label{48}
We have the following equation:
\[B_{-1}(g,\varphi) =  \CTE(g,\lambda(\varphi)) +c\RESE(g,\lambda(\tilde{\varphi})),\]
where $\CTE(g,\lambda(\varphi)) = \mathrm{CT}_{s = 1/2}\;E(g,s,\lambda(\varphi))$ and $\RESE(g,\lambda(\tilde{\varphi})) = \mathrm{Res}_{s = 1/2}E(g,s,\lambda(\tilde{\varphi}))$. Here $\lambda(\varphi)(g) := (\omega(g)\varphi)(0)\in I_1(1/2,\tilde{\chi}_V)$, $c\in\C$ is a constant independent of $g$ and $\tilde{\varphi}\in\mathcal{S}(V(\A))$ is an unspecified Schwartz function.
\end{thm}
\begin{proof}
See \cite[Theorem\;8.1,\;(ii),\;p.\;35]{gan_2014_the}.
\end{proof}
\subsubsection{Relevant complex variable Eisenstein series}
The Eisenstein series obtained in the Siegel Weil fromula define complex variable Eisenstein series. In this subsection we will introduce them and establish certain properties that will be relevant in the subsequent discussion.\\

Given $l\in \frac{1}{2}\Z$ we denote by $\Phi^l(g)\in I_1(1/2,\tilde{\chi}_V)$ the unique section of the induced representation satisfying that
$$\Phi^l(k_{\theta}) = e^{i l \theta},$$
when $k_{\theta} = \left(\begin{smallmatrix}\mathrm{cos}\theta& \mathrm{sin}\theta\\-\mathrm{sin}\theta&\mathrm{cos}\theta\end{smallmatrix}\right)\in \mathrm{SO}(2)$. The function $\Phi^l(g)$ is referred to as \textit{the section of weight} $l$ and it is in the image of the map $\lambda_{\infty}$ defined in \eqref{intermap2}. In fact, using the formulas of the Weil representation we find that $\Phi^l\left(g\right) = \lambda_{\infty}(\varphi^{\infty}_l)(g)$, where $\varphi^{\infty}_l\in \mathcal{S}(V(\R))$ is an Schwartz function satisfying that $\omega(k_{\theta})\varphi^{\infty}_l(x) = e^{i l \theta}\varphi^{\infty}_l(x)$. As a consequence of the Siegel-Weil formula, the following functions will be central to the discussion throughout this paper:
\begin{align*}E(\tau,s,l,\mu(\varphi_f)) &:= v^{-l/2}E(g_{\tau},s,\Phi^{l}\otimes \lambda(\varphi_f)),\\ \CTE(\tau,l,\mu(\varphi_f)) &:= v^{-l/2}\mathrm{CT}_{s = 1/2}E(g_{\tau},s,\Phi^{l}\otimes \lambda(\varphi_f)),\\
\RESE(\tau,l,\mu(\varphi_f)) &:= v^{-l/2}\mathrm{Res}_{s = 1/2}E(g_{\tau},s,\Phi^{l}\otimes \lambda(\varphi_f)),
\end{align*}
where $\varphi_f\in\mathcal{S}(V(\A_f))$.
\begin{prop}\label{107}
The $\mathrm{SL}_2-$Einsenstein series satisfy the following relation: 
$$-2iv^2\frac{\partial}{\partial\overline{\tau}}\left\{v^{\frac{-1}{2}(l+2)}E(g_{\tau},s,\Phi^{l+2}\otimes \lambda(\varphi_f))\right\} = \frac{1}{2}(s-l-1)v^{\frac{l-1}{2}}E(g_{\tau},s,\Phi^{l}\otimes \lambda(\varphi_f)),$$
and hence 
\begin{equation*}-2iv^2\frac{\partial}{\partial\overline{\tau}}\left\{E(\tau,s,l+2,\mu(\varphi_f))\right\} = \frac{1}{2}(s-l-1)E(\tau,s,l,\mu(\varphi_f)).
\end{equation*}
\end{prop}
\begin{proof}
It is \cite[(2.15),\;p.\;20]{kudla_2003_integral}.
\end{proof}
\begin{prop}\label{ConstantTermsEisenstein}
   The Eisenstein series $E(\tau,s,-1/2,\mu(\varphi_{\mu_0}))$ has a pole at $s = 1/2$. Moreover its Fourier constant term has residue:
   \[A_{-1}(\tau,-1/2,\mu(\varphi_{\mu_0}))_0 = v^{-3/2}\frac{3\sqrt{2}(-i)^{-1/2}}{\pi},\]
   and Laurent constant term:
    \[A_{0}(\tau,-1/2,\mu(\varphi_{\mu_0}))_0= v+v^{3/4}(-1)^{1/4}\sqrt{2}\pi\left(\frac{\log(v/16)+4}{2}-2\gamma\right).\]
\end{prop}
\begin{proof}
    Since $\varphi_{\mu_0}$ is unramified for all the primes $p$, we apply \cite[Theorem 2.4 (ii), p. 2282]{kudla_2010_eisenstein}.
\end{proof}
In subsequent sections of this paper we will also consider the Eisenstein series $E(\tau,s,3/2,\mu(\varphi_{\mu_j}))$. According to \cite[Corollary\;2.5,\;p.\;2283]{kudla_2010_eisenstein}, this function is holomorphic in the variable $s$. We denote by 
$$E(\tau,s,3/2,\mu(\varphi_{\mu_j})) = \sum_{m\in\Q}A(s,m,v,\mu(\varphi_{\mu_j}))q^m,$$
its Fourier series. Furthermore, we denote the Laurent expansion at $s = 1/2$ of each Fourier coefficient by
\begin{equation}\label{eqLaurent}A(s,m,v,\mu(\varphi_{\mu_j})) := a(m)+b(m,v,\mu(\varphi_{\mu_j}))(s-s_0)+\mathcal{O}((s-s_0)^2).\end{equation}
If it converges, we will use the following notation:
\[b(m,\varphi_{\mu_j}) := \lim_{T\to\infty}b(m,T,\mu(\varphi_{\mu_j})).\]
\section{Truncated Siegel-Weil formula}\label{sec4}
In this section we fix $(V,q)$ the rational quadratic space and the lattice $L$ as defined in \S\ref{modexample}. Recall that the bilinear form associated to $q$ is denoted by $(\cdot,\cdot)$. Let $H = O(V)$ and $G = \mathrm{Mp}_2$. Additionally, let $(W,\left<\cdot,\cdot\right>)$ denote the rational symplectic space of dimension $2$. By fixing a basis $e_1,e_2$ such that $\left<e_i,e_j\right> = \delta_{ij}$, we obtain the isomorphism $\Sp(W)\simeq \mathrm{SL}_2$.\\

Let $\Gaussr\in \mathcal{S}(V(\R))$ be the Gaussian associated to the quadratic space $(V,q)$ with base point $z_0 := i\in\mathcal{H}$. Throughout this section we will consider the Schwartz functions
\begin{equation}\label{schgauss}\varphi^{\infty}_{z_0,\mu_j} := \Gaussr\otimes\varphi_{\mu_j}\in\mathcal{S}(V(\A)),\end{equation}
where the functions $\varphi_{\mu_j}$ were introduced at the beginning of \S\ref{BorcherdsFormSubsec}. By the analysis done in the proof of \cite[Proposition\;5.3.1,\;p.\;45]{kudla_rallis}, the following integrals do not converge
\begin{equation}\label{eq58}\int_{[H]}\theta(g_{\tau},h,\varphi^{\infty}_{z_0,\mu_j})dh,\end{equation}
where we recall that 
$$g_{\tau}\in G(\R) = \left(\begin{pmatrix}v^{1/2}&uv^{1/2}\\0&v^{-1/2}\end{pmatrix},1\right).$$
The main goal of this section is to state an asymptotic formula for \eqref{eq58} i.e. we will isolate the terms of the theta function that diverge and we will compute the integral of the convergent ones. This computation is based on a manipulation of the second term identity of the Siegel-Weil formula developed in \cite{gan_2014_the}. 
\subsection{Factorization of the theta function}
Along this section, $p$ will be a place satisfying the hypothesis of \cite[p.\;209]{ichino_2001_on}. We will consider $\alpha := \alpha_{1,0,\eta}\in\mathcal{H}^{G}_{p}$ the Hecke operator described at the beginning of \S\ref{theregularizedsection}. Additionally, we will use the basis $<f_1,f_2,f_3>$ of $V$ as defined in \S\ref{modexample}, along with the basis $<e_1,e_2>$ of $W$ as specified above. Accordingly, we will adopt the notation $(a,b,c)^t = af_2+bf_1+cf_3$ and $(r,t) = re_1+te_2$.
\begin{prop}\label{mixed}
The following map is an isomorphism
\begin{align*}
    \mathcal{S}(V(\A))&\to  \mathcal{S}(\A)\otimes  \mathcal{S}(W(\A)),\\
    \varphi(x)&\mapsto \hat{\varphi}(x_0,u,v) := \int_{\A}\varphi\begin{pmatrix}x\\x_0\\u\end{pmatrix}\psi_{\A}(vx)dx,
\end{align*}
where $w := (u,v)\in W(\A)$, with $u,v\in \A$ and $\psi_{\A}$ is the adelic character used to define the Weil representation.
\end{prop}
\begin{proof}
See \cite[(5.3.2),\;p.\;45]{kudla_rallis}.
\end{proof}
\begin{defn}
\em
The isomorphism given by Proposition \ref{mixed} allows us to consider the representation \begin{align*}G(\A)\times H(\A)&\to \mathrm{Aut}(\mathcal{S}(\A)\otimes  \mathcal{S}(W(\A)))\\ (g,h)&\mapsto \left(\hat{\varphi}(x_0,u,v)\mapsto \int_{\A}\omega(g,h)\varphi\begin{pmatrix}x\\x_0\\u\end{pmatrix}\psi_{\A}(vx)dx\right),\end{align*} which is called the \textit{mixed model} of the Weil representation.
\em
\end{defn}
\begin{obs}\label{poissum}
By means of the partial Poisson summation formula stated in \cite[(5.3.4),\;p.\;45]{kudla_rallis}, the theta function  satisfies \[\theta(g,h,\varphi) = \sum_{\substack{x_0\in \Q\\w\in W(\Q)}}\hat{\varphi}(x_0,w).\]
\end{obs}
\begin{defn}\label{161}
\em
Let $\varphi\in\mathcal{S}(V(\A))$. Given a theta function 
$\theta(g,h,\varphi)$ for the dual reductive pair $(G,H)$, we define the \textit{divergent part} by
$$\mathrm{Div}(g,h,\varphi) := \sum_{\substack{x_0\in \Q}}\omega(g,h)\hat{\varphi}(x_0,0),$$
where $0$ represents $(0,0)\in W(\Q)$. Moreover, we define the \textit{convergent part} by
$$\mathrm{Conv}(g,h,\varphi) := \sum_{\substack{x_0\in \Q\\0\neq w\in W(\Q)}}\omega(g,h)\hat{\varphi}(x_0,w).$$
\em
\end{defn}
\begin{prop}\label{rapdec}
The convergent part $\mathrm{Conv}(g,h,\varphi)$ is rapidly decreasing and hence $\int_{[H]}\mathrm{Conv}(g,h,\varphi)dh$ is absolutely convergent for all $g\in G(\A)$.
\end{prop}
\begin{proof}
See the proof of \cite[Proposition\;5.3.1,\;p.\;45]{kudla_rallis}.
\end{proof}
\begin{obs}\label{convanddiv}
By Remark \ref{poissum} the previous definitions provide a well defined factorization for the theta function: 
$$\theta(g,h,\varphi) = \mathrm{Div}(g,h,\varphi)+\mathrm{Conv}(g,h,\varphi).$$
\end{obs}
\begin{lem}\label{divvanishes}
The regularized theta function satisfies 
\[\theta(g,h,\operatork\varphi) = \mathrm{Conv}(g,h,\operatork\varphi).\]
\end{lem}
\begin{proof}
The equality in the statement follows from the general equality provided in \cite[p. 210]{ichino_2001_on} which can be specialized to the case where $W$ has dimension $2$ and $V$ is defined as in \S\ref{modexample}.
\end{proof}
\begin{prop}\label{proptrans}
The following equality holds:
\[B_{-1}(g,\varphi) =  \int_{[H]}\mathrm{Conv}(g,h,\varphi)dh.\]
\end{prop}
\begin{proof}
By \eqref{equalityoperators}, we have that $\mathrm{Conv}(g,h,\operatork\varphi) = \mathrm{Conv}(g,h,\omega(\alpha_H)\varphi)$. Moreover, using the definition of the Hecke algebra on $\mathcal{S}(V(\A))$ via the Weil representation as described in \eqref{HeckeAlgWeil}, we obtain 
\[\mathrm{Conv}(g,h,\operatork\varphi) = \int_{H(\Q_p)}\alpha_H(h_p)\mathrm{Conv}(g,hh_p,\varphi)dh_p.\]
Applying Proposition \ref{divvanishes} and the above equality to $B_{-1}(g,\varphi)$ we obtain
\[B_{-1}(g,\varphi) = \mathrm{Res}_{s = 1/2}\left(\frac{1}{2P_{1,1}(s)}\int_{[H]}\int_{H(\Q_p)}\alpha_H(h_p)\mathrm{Conv}(g,hh_p,\varphi)E(h,s)dh_pdh\right).\]
By Proposition \ref{rapdec} the function $\mathrm{Conv}(g,h,\varphi)$ is rapidly decreasing in the variable $h$, hence, we do a change of variables of the form $h = hh_v$, obtaining 
\begin{equation}\label{auxPropB-1}B_{-1}(g,\varphi) = \mathrm{Res}_{s = 1/2}\left(\frac{1}{2P_{1,1}(s)}\int_{[H]}\mathrm{Conv}(g,h,\varphi)\int_{H(\Q_p)}\alpha_H(h_v)E(hh_v^{-1},s)dh_vdh\right).\end{equation}
By equations \eqref{ActionHeckeAut} and \eqref{eigenEisenstein}, we have 
\[\int_{H(\Q_p)}\alpha_H(h_v)E(hh_v^{-1},s)dh_v = \alpha_H* E(h,s) = P_{1,1}(s)E(h,s).\]
Plugging the previous equality in \eqref{auxPropB-1} we conclude that
\[B_{-1}(g,\varphi) = \mathrm{Res}_{s = 1/2}\left(\frac{1}{2}\int_{[H]}\mathrm{Conv}(g,h,\varphi)E(h,s)dh\right) = \int_{[H]}\mathrm{Conv}(g,h,\varphi)dh,\]
where the last equality follows because $\mathrm{Res}_{s = 1/2}E(h,s)=2$.
\end{proof}
\begin{thm}\label{swcor}
There exists a Schwartz function $\tilde{\varphi}_j\in\mathcal{S}(V(\A))$ such that the convergent part satisfies
\[\int_{[H]}\mathrm{Conv}(g,h,\varphi^{\infty}_{z_0,\mu_j})dh = \CTE(g,\lambda(\varphi^{\infty}_{z_0,\mu_j}))+c\RESE(g,\lambda(\tilde{\varphi}_j)),\]
where $c\in\C$ is a constant independent of $g$. Moreover, for any $k_{\theta} = \left(\begin{smallmatrix}cos(\theta)&sin(\theta)\\-sin(\theta)&cos(\theta)\end{smallmatrix}\right) \in \mathrm{SL}_2(\R)$ the function $\tilde{\varphi}_j$ satisfies the property $\omega(k_{\theta})\tilde{\varphi}_j(x) = e^{-\frac{ i\theta}{2}}\tilde{\varphi}_j(x)$.
\end{thm}
\begin{proof}
The equality of the statement follows by Theorem \ref{48} and Proposition \ref{proptrans}.\\

By Lemma \ref{weightGaussian}, the functions
$\int_{[H]}\mathrm{Conv}(g,h,\varphi^{\infty}_{z_0,\mu_j})dh$ and $\CTE(g,\lambda(\varphi^{\infty}_{z_0,\mu_j}))$ transform under the action of $k_{\theta} = \left(\begin{smallmatrix}\mathrm{cos}\theta&\mathrm{sin}\theta\\ -\mathrm{sin}\theta&\mathrm{cos}\theta\end{smallmatrix}\right)$ by multiplication by $e^{-\frac{i\theta}{2}}$. This property also extends to $A_{-1}(g,\lambda(\tilde{\varphi}_j))$ due to the equality of the statement. Since the space of sections $\Phi$ with this transformation property is one dimensional, as established in \cite[(4.18),\;p.\;39]{kudla_2003_integral}, we obtain that the Schwartz function $\tilde{\varphi}_j\in\mathcal{S}(V(\R))$ satisfies $\omega(k_{\theta})\tilde{\varphi}_j(x) = e^{-\frac{i\theta}{2}}\tilde{\varphi}_j(x)$.
\end{proof}
\begin{cor}\label{swcor2}
Let $K_f< \mathrm{SL}_2(\A_f)$ be a compact open subgroup fixing $\varphi^{\infty}_{z_0,\mu_j}$ for some $j$. Then
\[\int_{[H]}\mathrm{Conv}(g,h,\varphi^{\infty}_{z_0,\mu_j})dh = \CTE(g,\lambda(\varphi^{\infty}_{z_0,\mu_j}))+c\RESE (g,\lambda(\tilde{\varphi}_j)),\]
with $\RESE (g,\lambda(\tilde{\varphi}_j))$ a $K_f-$invariant function.
\end{cor}
\begin{proof}
Since $\operatork$ commutes with the action of the Weil representation, the regularized theta function $\theta(g,h,\operatork\varphi^{\infty}_{z_0,\mu_j})$ is right $K_f-$invariant. Furthermore, using Lemma \ref{divvanishes} we find that the the function $\mathrm{Conv}(g,h,\varphi^{\infty}_{z_0,\mu_j})$ is also $K_f-$invariant. The intertwining map $\lambda$ preserves the $K_f-$invariance of $\varphi$, then it follows that the Eisenstein series $E(g,s,\lambda(\varphi^{\infty}_{z_0,\mu_j}))$ and its Laurent constant term are also $K_f-$invariant. Therefore the equation given in Theorem \ref{swcor} implies the $K_f$-invariance for the function $\RESE (g,\lambda(\tilde{\varphi}_j))$.
\end{proof}
\section{Integral of Borcherds forms}
As in the previous section we fix $(V,q)$ and $L$ the quadratic space and lattice defined in \S\ref{modexample}. Furthermore $H = O(V)$ and $G = \mathrm{Mp}_2$. Throughout this section, we will denote by \[f(\tau) = \sum_{j = 0}^1f_{\mu_j}(\tau)\varphi_{\mu_j} =  \sum_{\substack{n\in\Z\\j \in\{0,1\}}}c_{\mu_j}(n)q^n\varphi_{\mu_j}\in M_{1/2,L}^!,\] the Fourier expansion of a vector-valued weakly holomorphic modular form. For $1\leq\hat{T}\in\R$, define the truncated fundamental domain of the modular curve without level as \[X^{mod,{\hat{T}}} := \{z = x+iy\in \mathcal{H},\;s.t.\;|z|\geq 1,\;|x|\leq 1/2,\;y<\hat{T}\},\]
and let 
\[\widehat{X^{mod,\hat{T}}} := \{z = x+iy\in \mathcal{H},\;|z|\geq 1,\;|x|\leq 1/2,\;y>\hat{T}\},\]
denote its complement. The goal of this section is to compute the following integral:
\begin{equation}\label{eq204}\int_{X^{mod,\hat{T}}}\log||\Psi(f)(z)||_{Pet}d\mu(z),\end{equation}
where $\Psi(f)(z)$ is the Borcherds lift of $f$.
Using the definition of the Petersson norm it is straightforward that
\begin{align}\int_{X^{mod,\hat{T}}}\log||\Psi(f)(z)||_{Pet}d\mu(z) &= \int_{X^{mod,\hat{T}}}\log|\Psi(f)(z)y^{c_{\mu_0}(0)/2}|d\mu(z) \nonumber\\
\label{eq203} &= \int_{X^{mod,\hat{T}}}\log|\Psi(f)(z)|d\mu(z)+\frac{c_{\mu_0}(0)}{2}\int_{X^{mod,\hat{T}}}\log|y|d\mu(z).\end{align}
The result \cite[Theorem\;13.3,\;p.\;48]{borcherds_1998_automorphic} shows the following relation:
$$\log|\Psi(f)(z)| = -\frac{\Phi(f)(z)}{4}- \frac{c_{\mu_0}(0)}{2}\left(\log|y|-\gamma/2+\log\sqrt{2\pi}\right),$$
where $\Phi(f)(z)$ is the singular theta lift of $f$.
Plugging the previous equality in \eqref{eq203} we find
$$\int_{X^{mod,\hat{T}}}\log||\Psi(f)(z)||_{Pet}d\mu(z) =  -\frac{1}{4}\int_{X^{mod,\hat{T}}}\Phi(f)(z) d\mu(z)-\frac{c_{\mu_0}(0)\vol(X^{mod,\hat{T}})}{2}\left(-\gamma/2+\log\sqrt{2\pi}\right).$$
Therefore, our goal will be achieved by computing the following integral:
\begin{align}\label{eq113}\int_{X^{mod,\hat{T}}}\Phi(f)(z) d\mu(z) &= \int_{X^{mod,\hat{T}}}\left(\int_{X^{mod}}^{\bullet}\left<f(\tau),\Theta^{Sieg}_L(\tau,z)\right>d\mu(\tau)\right)d\mu(z) \\  &= \sum_{j = 0}^1\int_{X^{mod,\hat{T}}}\left(\int_{X^{mod}}^{\bullet}f_{\mu_j}(\tau)\theta^{Sieg}_{\mu_j}(\tau,z)d\mu(\tau)\right)d\mu(z).\nonumber\end{align}

This section is organized as follows. A key step in computing \eqref{eq113} is the integral of the Siegel theta function over the truncated fundamental domain of the modular curve without level, which is the focus of \S\ref{136}. To compute this, we first apply Theorem \ref{swcor} along with ideas from \cite[Proposition 4.17, p. 44]{kudla_2003_integral}. However, unlike in \cite{kudla_2003_integral}, this computation alone is not sufficient. In \S\ref{IntegralSingularThetaSec}, we outline the approach for computing the integral of the singular theta lift, starting by recalling the strategy used in \cite{kudla_2003_integral} and highlighting the problems that prevent us from directly applying it in our case. We then explain how to adapt this previous strategy to the current setting. The approach relies on the following factorization of the integral: let us denote by \[\theta^{Sieg}_{\mu_j}(\tau,z) = \sum_{\substack{\lambda\in L+\mu_j}}\theta^{Sieg}_{\mu_j}(\tau,z)_{q(\lambda)},\] 
the Fourier expansion of $\theta^{Sieg}_{\mu_j}(\tau,z)$ with respect to the variable $\tau$. The constant Fourier coefficient \[\theta^{Sieg}_{\mu_j}(\tau,z)_{0} := \sum_{\substack{\lambda\in L+\mu_j\\q(\lambda) = 0}}\theta^{Sieg}_{\mu_j}(\tau,z)_{\lambda},\] does not have exponential decay when $v\to\infty$. Therefore, in order to apply the previously mentioned truncated Siegel-Weil formula, we will factor the integrals 
\begin{equation}\label{eq135}\lim_{T\to\infty}\int_{\mathcal{F}^T}f_{\mu_j}(\tau)\theta^{Sieg}_{\mu_j}(\tau,z)v^{1-\sigma}d\mu(\tau),\end{equation}
as the sum of two terms; the \textit{ordinary case}:
\[\lim_{T\to\infty}\int_{\mathcal{F}^T}f_{\mu_j}(\tau)\sum_{\substack{\lambda\in L+\mu_j\\q(\lambda)\neq 0}}\theta^{Sieg}_{\mu_j}(\tau,z)_{q(\lambda)}v^{1-\sigma}d\mu(\tau),\]
and the \textit{limit case}:
\[\lim_{T\to\infty}\int_{\mathcal{F}^T}f_{\mu_j}(\tau)\theta^{Sieg}_{\mu_j}(\tau,z)_{0}v^{1-\sigma}d\mu(\tau).\]
The first integrals are studied in \S\ref{216} using the geometric version of the truncated Siegel-Weil formula. The second integrals will be studied in \S\ref{219} using the truncated Rankin-Selberg method of \cite{Don1982TheRM}.
\subsection{Geometric version of the truncated Siegel-Weil formula.}\label{136}
In the previous section we computed the integral of the convergent terms of the adelic theta function, see Theorem \ref{swcor}. This subsection connects this result with the complex geometry point of view of Borcherds \cite{borcherds_1998_automorphic}, showing a truncated version of the classical Siegel-Weil formula for the modular curve.\\

Throughout this subsection we fix $K^H(\A_f) = \prod_{p\nmid \infty}H(\Z_p)$ an open compact subgroup of $H(\A_f)$. Let us note that $K^H(\A_f)$ fixes the lattices $\hat{\Z}_f^3 := \Z^3\otimes_{\Z}\A_f$ and $\hat{\Z}_f\oplus \frac{1}{2}\hat{\Z}_f\oplus \hat{\Z}_f := (\Z\oplus\frac{1}{2}\Z\oplus\Z)\otimes_{\Z}\A_f$. Recall the notation \begin{equation}\label{formulasfinite}\varphi_{\mu_0} = \mathrm{char}_{\hat{\Z}_f^3}\;\;and\;\; \varphi_{\mu_1} = \mathrm{char}_{\hat{\Z}_f\oplus \frac{1}{2}\hat{\Z}_f\oplus \hat{\Z}_f}.\end{equation} We then set $\varphi^{\infty}_{z_0,\mu_0} := \Gaussr\otimes\varphi_{\mu_0}$ and $\varphi^{\infty}_{z_0,\mu_1} := \Gaussr\otimes\varphi_{\mu_1}$.
\begin{lem}\label{88}
The divergent parts $\mathrm{Div}(g,h,\Gaussi)$ are right $S(O(2)\oplus O(1))\times K^H(\A_f)-$invariant.
\end{lem}
\begin{proof}
We recall
\begin{align*}\mathrm{Div}(g,h,\varphi^{\infty}_{z_0,\mu_j}) &= \sum_{x_0\in\Q}\omega(g,h)\widehat{\varphi^{\infty}_{z_0,\mu_j}}(x_0,0,0)  \\ &= \sum_{x_0\in \Q}\int_{\A}\omega(g,h)\varphi^{\infty}_{z_0,\mu_j}\begin{pmatrix}x\\x_0\\0\end{pmatrix}dx.\end{align*}
Let us denote by $$k = k_{\R}\times k_{f}\in S(O(2)\oplus O(1))\times K^H(\A_f).$$
It is straightforward that 
\begin{align*}\int_{\A}\omega(k)\varphi^{\infty}_{z_0,\mu_j}\begin{pmatrix}x\\x_0\\0\end{pmatrix}dx &= \int_{\R}\omega(k_{\R})\Gaussr\begin{pmatrix}x\\x_0\\0\end{pmatrix}dx\cdot\int_{\A_f}\omega(k_f)\varphi_{\mu_j}\begin{pmatrix}x_f\\x_0\\0\end{pmatrix}dx_f \\ &= \int_{\A}\varphi^{\infty}_{z_0,\mu_j}\begin{pmatrix}x\\x_0\\0\end{pmatrix}dx,\end{align*}
where the latter equality follows since $\Gaussr$ is $S(O(2)\oplus O(1))-$invariant and $\varphi_{\mu_j}$ is $K^H(\A_f)-$invariant by Lemma \ref{121}. Therefore the function $\widehat{\varphi^{\infty}_{z_0,\mu_j}}$ is $S(O(2)\oplus O(1))\times K^H(\A_f)-$invariant, implying that $\mathrm{Div}(g,h,\varphi^{\infty}_{z_0,\mu_j})$ is $S(O(2)\oplus O(1))\times K^H(\A_f)-$invariant.
\end{proof}
\begin{lem}\label{invarianza}
The regularized theta functions 
$\theta(g,h,\operatork\Gaussi)$
are right $(O(2)\oplus O(1))\times K^H(\A_f)-$ invariant.
\end{lem}
\begin{proof}
Since $\operatork$ commutes with the Weil representation the proof is analogous to Lemma \ref{88}.
\end{proof}
\begin{prop}\label{114}
The integral of the theta function satisfies the following equality: 
$$\int_{[H]}\theta(g_{\tau},h,\operatork\Gaussi)dh  = \frac{1}{2}\int_{[\mathrm{SO}(V)]}\theta(g_{\tau},h,\operatork\Gaussi)dh.$$
\end{prop}
\begin{proof}
Let us define the map 
\begin{align*}I:\mathcal{S}(V(\A))&\to \C,\\\varphi&\mapsto \int_{[SO(V)]}\theta(\mathrm{id},h,\omega(\alpha)\varphi)dh,\end{align*}
where $\mathrm{id}$ is the identity element of $G(\A)$. According to \cite[Lemma 1.3, p. 208]{ichino_2001_on}, the above integral converges absolutely. Since $\omega(\alpha)$ commutes with the action of the Weil representation and the group $SO(V)$ is unimodular, we obtain
\[I\in \mathrm{Hom}_{SO(V)(\A)}(\mathcal{S}(V(\A)),\C),\]
where the action of $SO(V)(\A)$ is trivial on $\C$. Moreover, by \cite[Proposition 4.2, p. 37]{kudla_2003_integral}, the group $C(\A) := O(V)(\A)/SO(V)(\A)$ acts trivially on this space of functionals (denoted as $C(\A_f)$ in Kudla's notation).\\ Let us fix a Haar measure $dc$ on $\mathrm{C}(\A)$ such that $\mathrm{vol}(C(\A),dc) = 1$. We have the following equalities
\[\int_{[H]}\theta(g_{\tau},h,\omega(\alpha)\varphi) dh = \int_{C(\Q)\setminus C(\A)}\int_{[SO(V)]}\theta(g_{\tau},ch,\omega(\alpha)\varphi) dhdc = \frac{1}{2}\int_{[SO(V)]}\theta(g_{\tau},h,\omega(\alpha)\varphi) dh,\]
where we have used that $\mathrm{vol}(C(\Q)\setminus C(\A)) = \frac{1}{2}$. 
\end{proof}
\begin{cor}\label{103}
The convergent part satisfies the following equality:
$$\int_{[H]}\mathrm{Conv}(g_{\tau},h,\Gaussi) dh =  \frac{1}{2}\int_{[\mathrm{SO}(V)]}\mathrm{Conv}(g_{\tau},h,\Gaussi)dh.$$
\end{cor}
\begin{proof}
By Lemma \ref{divvanishes}, we have 
\[\mathrm{Conv}(g_{\tau},h,\Gaussi) = \theta(g_{\tau},h,\omega(\alpha)\Gaussi).\]
Thus, the preceding Proposition establishes the equality in the statement.
\end{proof}
\begin{prop}\label{68}
We obtain $$\int_{[\mathrm{SO}(V)]}\mathrm{Conv}(g_{\tau},h,\Gaussi)dh = \frac{-\vol(K^{H}(\A_f))}{2}\int_{\mathrm{SL}_2(\Z)\setminus \mathcal{H} }\mathrm{Conv}(g_{\tau},h_1,\Gaussieva)d\mu(z),$$ where $d\mu(z) = \frac{dxdy}{y^2}$. 
\end{prop}
\begin{proof}
Let us set the notation $\tilde{\Hgrp} =\mathrm{GSpin}_V$. The strong approximation Theorem shows that
$$\tilde{\Hgrp}(\A) = \bigcup_{t\in T}\tilde{\Hgrp}(\Q)\tilde{\Hgrp}(\R)^+h_t\mathrm{K}$$
where $h_t\in \tilde{\Hgrp}(\A)$ and $\tilde{\Hgrp}(\R)^+$ is the connected component of the identity of $\tilde{\Hgrp}(\R)$. The modular curve $\mathrm{SL}_2(\Z)\setminus \mathcal{H}$ has one connected component, then $T = \{1\}$. Since the functions $\Gausszero$ and $\Gaussone$ are $S(O(2)\times O(1))\times K^{H}(\A_f)-$invariant, the proof is analogous to \cite[Proposition\;4.17,\;p.\;44]{kudla_2003_integral}.
\end{proof}
\begin{cor}\label{202}
It holds that
$$\int_{[\mathrm{SO}(V)]}\mathrm{Conv}(g_{\tau},h,\Gaussi)dh = \frac{2}{\vol(X^{mod})}\int_{\mathrm{SL}_2(\Z)\setminus \mathcal{H} }\mathrm{Conv}(g_{\tau},h_1,\Gaussieva)d\mu(z).$$
\end{cor}
\begin{proof}
The result follows by Proposition \ref{68} and by \cite[Remark\;4.18,\;p.\;46]{kudla_2003_integral}.
\end{proof}
\begin{thm}\label{135}
The integral of the convergent term over the modular curve satisfies 
$$\int_{\mathrm{SL}_2(\Z)\setminus \mathcal{H} }\mathrm{Conv}(g_{\tau},h_1,\Gaussieva)d\mu(z) =\vol(X^{mod})\left[\CTE(g_{\tau},\lambda(\Gaussieva))+c\RESE (\gtau,\lambda(\tilde{\varphi}_j))\right],$$
where $\tilde{\varphi}_j := \tilde{\varphi}_{j,\infty}\otimes\varphi^{\tilde{L}}_j\in \mathcal{S}(V(\A))$ is a Schwartz function satisfying $\omega(k_{\theta})\tilde{\varphi}_j(x) = e^{-\frac{i\theta}{2}}\tilde{\varphi}_j(x)$ for $k_{\theta} = \left(\begin{smallmatrix}cos(\theta)&sin(\theta)\\-\sin(\theta)&cos(\theta)\end{smallmatrix}\right) \in \mathrm{SL}_2(\R)$.
\end{thm}
\begin{proof}
Using Corollaries \ref{103} and \ref{202} we obtain 
\begin{equation}\label{eq221}\int_{\mathrm{SL}_2(\Z)\setminus \mathcal{H} }\mathrm{Conv}(g_{\tau},h_1,\Gaussieva)d\mu(z) = \vol(X^{mod})\int_{[H]}\mathrm{Conv}(g_{\tau},h,\Gaussi)dh.\end{equation}
Moreover, applying Theorem \ref{swcor} and Corollary \ref{swcor2} on the right hand side of \eqref{eq221}, the result follows.
\end{proof}
\begin{defn}\label{defconvdiv}
\em
We denote by
\[\vartheta(\tau,z,\mu_j) := v^{1/4}\theta(g_{\tau},h_1,\Gaussieva),\] \[\mathbf{Conv}(\tau,z,\mu_j) := v^{1/4}\mathrm{Conv}(\gtau,h_1,\Gaussieva),\; \mathbf{Div}(\tau,z,\mu_j) := v^{1/4}\mathrm{Div}(\gtau,h_1,\Gaussieva).\] Recall that the element $h_1\in \tilde{H}(\A)$ was defined in the proof of Proposition \ref{68}. Furthermore we denote their constant Fourier coefficients with respect to $\tau$ by $\vartheta(v,z,\mu_j)_0$, $\mathbf{Conv}(v,z,\mu_j)_0$ and $\mathbf{Div}(v,z,\mu_j)_0$ respectively.  
\end{defn}
\begin{thm}\label{164}
Let $\hat{T}>1$, the theta function satisfies the following equality:
\begin{align*}\int_{X^{mod,\hat{T}}}\vartheta(\tau,z,\mu_j)d\mu(z)=&\vol(X^{mod})\Big[\CTE(\tau,-1/2,\mu(\varphi_{\mu_j}))+c\RESE(\tau,s,-1/2,\mu(\tilde{\varphi}_j))\Big] \\&-\int_{\widehat{X^{mod,\hat{T}}}}\mathbf{Conv}(\tau,z,\mu_j)d\mu(z) +\int_{X^{mod,\hat{T}}}\mathbf{Div}(\tau,z,\mu_j)d\mu(z).\end{align*}
\end{thm}
\begin{proof}
We factor $\vartheta(\tau,z,\mu_j) $ into its divergent and convergent parts, as we did in Remark \ref{convanddiv}. We obtain
\begin{align}\label{eq136}\int_{X^{mod,\hat{T}}}\vartheta(\tau,z,\mu_j)d\mu(z) =& \int_{X^{mod,\hat{T}}}\mathbf{Conv}(\tau,z,\mu_j)d\mu(z)\\ &+\int_{X^{mod,\hat{T}}}\mathbf{Div}(\tau,z,\mu_j)d\mu(z).\nonumber\end{align}
Since $\mathbf{Conv}(\tau,z,\mu_j)$ is integrable
\begin{align}\label{aux11}\int_{X^{mod,\hat{T}}}\mathbf{Conv}(\tau,z,\mu_j)d\mu(z)  =& \int_{X^{mod}}\mathbf{Conv}(\tau,z,\mu_j)d\mu(z)\\ &-\int_{\widehat{X^{mod,\hat{T}}}}\mathbf{Conv}(\tau,z,\mu_j)d\mu(z).\nonumber\end{align}
Using Theorem \ref{135} in \eqref{aux11} and then plugging the result into \eqref{eq136}, we obtain the formula.
\end{proof}
\begin{cor}\label{237}
The constant term of the theta function satisfies the following equality:
\begin{align*}\int_{X^{mod,\hat{T}}}\vartheta(v,z,\mu_0)_0d\mu(z) =&\vol(X^{mod})\Big[\CTE(v,-1/2,\mu(\varphi_{\mu_0}))_0+c\RESE(v,-1/2,\mu(\tilde{\varphi}_{0}))_0\Big]\\ &-\int_{\widehat{X^{mod,\hat{T}}}}\mathbf{Conv}(v,z,\mu_0)_0d\mu(z)+\int_{X^{mod,\hat{T}}}\mathbf{Div}(v,z,\mu_0)_0d\mu(z),\end{align*}
where $\CTE(v,-1/2,\mu(\varphi_{\mu_0}))_0$ and $\RESE(v,-1/2,\mu(\tilde{\varphi}_{0}))_0$ are the constant Fourier coefficients with respect to $\tau$ of $\CTE(\tau,-1/2,\mu(\varphi_{\mu_0}))$ and $\RESE(\tau,-1/2,\mu(\tilde{\varphi}_{0}))$ respectively.
\end{cor}
\begin{proof}
We have
$$\int_{X^{mod,\hat{T}}}\vartheta(v,z,\mu_0)_0d\mu(z) = \int_{X^{mod,\hat{T}}}\left(\int_{-1/2}^{1/2}\vartheta(u+iv,z,\mu_0)du\right)d\mu(z).$$
Since $X^{mod,\hat{T}}$ is compact and $\vartheta(\tau,z,\mu_0)$ is continuous in both coordinates, we can apply Fubini's Theorem. Then
\begin{align*}\int_{X^{mod,\hat{T}}}&\left(\int_{-1/2}^{1/2}\vartheta(u+iv,z,\mu_0)du\right)d\mu(z)= \int_{-1/2}^{1/2}\left(\int_{X^{mod,\hat{T}}}\vartheta(u+iv,z,\mu_0)d\mu(z)\right)du.\end{align*}
We apply Theorem \ref{164}, obtaining
\begin{align}\label{eq240}\int_{-1/2}^{1/2}&\left(\int_{X^{mod,\hat{T}}}\vartheta(u+iv,z,\mu_0)d\mu(z)\right)du  =\vol(X^{mod})\Big[\CTE(v,-1/2,\mu(\varphi_{\mu_0}))_0+c\RESE(v,-1/2,\mu(\tilde{\varphi}_{0}))_0\Big] \\ &-\int_{-1/2}^{1/2}\int_{\widehat{X^{mod,\hat{T}}}}\mathbf{Conv}(u+iv,z,\mu_0)d\mu(z)dv +\int_{-1/2}^{1/2}\int_{X^{mod,\hat{T}}}\mathbf{Div}(u+iv,z,\mu_0)d\mu(z)dv.\nonumber\end{align}
Using Fubini's Theorem one more time 
\begin{align}\label{eq239}\int_{-1/2}^{1/2}&\int_{X^{mod,\hat{T}}}\mathbf{Div}(u+iv,z,\mu_0)d\mu(z)du  =\int_{X^{mod,\hat{T}}}\left(\int_{-1/2}^{1/2}\mathbf{Div}(u+iv,z,\mu_0)du\right)d\mu(z) \\&= \int_{X^{mod,\hat{T}}}\mathbf{Div}(v,z,\mu_0)_0d\mu(z).\nonumber\end{align}
Furthermore, analogously to \eqref{eq239}, we obtain  \begin{equation}\label{eq2411}\int_{-1/2}^{1/2}\int_{\widehat{X^{mod,\hat{T}}}}\mathbf{Conv}(u+iv,z,\mu_0)d\mu(z)du = \int_{\widehat{X^{mod,\hat{T}}}}\mathbf{Conv}(v,z,\mu_0)_0d\mu(z).\end{equation}
Using the equalities given in \eqref{eq239} and \eqref{eq2411} on the right hand side of the equality \eqref{eq240}, the result follows. 
\end{proof}
\subsection{Integral of the singular theta lift.}\label{IntegralSingularThetaSec}
Let $f\in M^!_{1/2,L}$, the goal of this subsection is to outline the strategy for computing the integral
$$\int_{X^{mod,\hat{T}}}\left(\int_{X^{mod}}^{\bullet}\left<f(\tau),\Theta^{Sieg}(\tau,z)\right>d\mu(\tau)\right)d\mu(z),$$
which appears in equation \eqref{eq113}. As noted earlier in this section, we will factor this integral in terms of the Fourier coefficients of the Siegel theta function.
\subsubsection{Sketch of the computation for higher dimensional Shimura varieties}
Before starting the computation we will recall the method of \cite{kudla_2003_integral}. Although it does not apply to our case, the strategy used in \cite{kudla_2003_integral} will be usefull for our goal. Let $(V^{(p,2)},Q)$ be either an isotropic rational quadratic space of signature $(p,2)$ with $p\geq 3$ or an anisotropic rational quadratic space and set $l' := \frac{p-2}{2}$. Let us fix an even lattice $M$ of $V$ so that $Q|_{M}\in \Z$. In \cite{kudla_2003_integral} the author studies the integrals
$$\int_{X^{(p,2)}}\left(\int_{X^{mod}}^{\bullet}\left<f(\tau),\Theta^{Sieg}_M(\tau,z)\right>d\mu(\tau)\right)d\mu(z).$$
Here $X^{(p,2)}$ denotes the Shimura variety associated to the algebraic group $\mathrm{GSpin}_{V^{(p,2)}}$ with lattice $M$ and $\Theta^{Sieg}_M(\tau,z) := \sum_{\mu\in M'(\A_f)/M(\A_f)}\tilde{\vartheta}(\tau,z,\mu)\mathrm{char}_{\mu+M}$ is the higher dimensional vector valued Siegel theta function, with $\tilde{\vartheta}(\tau,z,\mu)$ defined as the function $\theta_{\mu}(\tau,z;0,0)$ in \cite[(2.1), p. 40]{bruinier_2007_borcherds}. According to \cite[\S3,\;p.\;24]{kudla_2003_integral}, we have
\begin{align}\int_{X^{(p,2)}}&\left(\int_{X^{mod}}^{\bullet}f_{\mu}(\tau)\tilde{\vartheta}(\tau,z,\mu)vd\mu(\tau)\right)d\mu(z)\nonumber \\\label{eq114}&=\int_{X^{(p,2)}}\left(\int_{X^{mod}}^{\bullet}f_{\mu}(\tau)\tilde{\vartheta}(\tau,z,\mu)d\mu(\tau)\right)d\mu(z)   \\\label{eq1142} &= \mathrm{CT}_{\sigma = 0}\lim_{T\to\infty}\int_{\mathcal{F}^T}f_{\mu}(\tau)\left(\int_{X^{(p,2)}}\tilde{\vartheta}(\tau,z,\mu)d\mu(z)\right)v^{-\sigma}d\mu(\tau),\end{align}
After this, the author applies the Siegel-Weil formula and unfolds the integral using the resulting Eisenstein series. Although the proof of the equality between \eqref{eq114} and \eqref{eq1142} does not extend to the case of a rational isotropic quadratic space of signature $(1,2)$, i.e. the modular curve case, we will still provide an overview of the proof, and this will be the main goal of this subsection. We will highlight which parts are relevant for our purposes and identify the propositions that do not apply in this particular case.
\begin{prop}\label{fourierexp}
Let $\theta(g,h,\varphi)$ be the theta function associated to any dual reductive pair of the form $(\Sp_1,O(V^{(p,2)}))$ with $V^{(p,2)}$ a rational quadratic space of signature $(p,2)$. Given $\beta\in \Q$, the $\beta-$Fourier coefficient of $\theta(g,h)$ with respect to $g$ is equal to 
$$\theta(g,h,\varphi)_{\beta} = \sum_{\substack{x\in\Q^{p+2}\\((x,x)) = 2\beta}}\omega(g,h)\varphi(x).$$
The Fourier expansion of $\tilde{\vartheta}(\tau,z,\mu)$ with respect to $\tau$ is given by 
$$\tilde{\vartheta}(\tau,z,\mu) = v\sum_{\lambda\in \mu+M}e^{-2\pi(vQ(\lambda_z)-vQ(\lambda_{z^{\perp}}))}e^{2\pi iQ(\lambda)u}.$$
Furthermore 
\[\left<f(\tau),\Theta_M^{Sieg}(\tau,z)\right> = \sum_{\mu\in M'/M}\sum_{m\in \Q}c_{\mu}(-m)\sum_{\substack{\lambda\in \mu + M \\Q(\lambda) = m}}e^{-2\pi v(Q(\lambda_z)-Q(\lambda_{z^{\perp}}))}.\]
\end{prop}
\begin{proof}
We find the first assertion in  \cite[(5.1),\;p.\;512]{kudla_1992_on} and the second one in \cite[p.\;48]{bruinier.Bor}. 
\end{proof}
Using the previous Proposition we may factor our theta function as follows: 
\begin{align}\label{eq133}\tilde{\vartheta}(\tau,z,\mu) =& C_{00}(\tau,z,\mu)+ C_0(\tau,z,\mu)+C_1(\tau,z,\mu)+C_2(\tau,z,\mu)\\ =& v\delta_{\mu,0}+v\sum_{\substack{0\neq \lambda\in \mu+M\\Q(\lambda) = 0}}e^{-2\pi(vQ(\lambda_z)-vQ(\lambda_{z^{\perp}}))}\nonumber\\&+v\sum_{\substack{\lambda\in  \mu+M\\Q(\lambda) > 0}}e^{-2\pi(vQ(\lambda_z)-vQ(\lambda_{z^{\perp}}))}e^{2\pi iQ(\lambda)u} \nonumber\\ &+v\sum_{\substack{\lambda\in  \mu+M\\Q(\lambda) < 0}}e^{-2\pi(vQ(\lambda_z)-vQ(\lambda_{z^{\perp}}))}e^{2\pi iQ(\lambda)u},\nonumber
\end{align}
where $\delta_{\mu,0} = 1$ if and only if $\mu\in M$ and $0$ otherwise. The functions $C_{00}(\tau,z,\mu),\;C_0(\tau,z,\mu),\;C_1(\tau,z,\mu)$ and $C_2(\tau,z,\mu)$ correspond to the terms on the right hand side of the equality.
\begin{obs}\label{130}
In the formula \eqref{eq133} the terms $C_1(\tau,z,\mu)$ are the positive Fourier coefficients of $\vartheta(\tau,z,\mu)$ with respect to $\tau$, $C_2(\tau,z,\mu)$ are the negative Fourier coefficients of $\vartheta(\tau,z,\mu)$ with respect to $\tau$ and $C_{00}(\tau,z,\mu)+C_0(\tau,z,\mu)$ is the $0-$th Fourier coefficient of $\vartheta(\tau,z,\mu)$ with respect to $\tau$. The motivation to split the $0-$th Fourier coefficient into two terms is the asymptotic behaviour with respect to $\tau = u+iv$. The function 
$$C_{0}(\tau,z,\mu) = v\sum_{\substack{\lambda\in  \mu+M\\Q(\lambda) = 0\\ \lambda\neq 0 }}e^{-2\pi(vQ(\lambda_z)-vQ(\lambda_{z^{\perp}}))},$$
decreases as $e^{-v}$ when $v\to\infty$. By contrast, when $\mu\in M$, the function $C_{00}(\tau,z,\mu) = v$ is not exponentially decreasing. In fact, the term $C_{00}(\tau,z,\mu)$ is the reason why $\mathrm{CT}_{\sigma = 0}$ is needed to state the singular theta lift, see \cite{borcherds_1998_automorphic}.
\end{obs}
To establish the equality between \eqref{eq114} and \eqref{eq1142}, the author in \cite[\S;3,;p.;24]{kudla_2003_integral} factors the integral \eqref{eq114} based on the factorization provided in \eqref{eq133}:
\begin{align*}\mathrm{CT}_{\sigma = 0}\int_{X^{(p,2)}}&\left(\lim_{T\to\infty}\int_{\mathcal{F}^T}f_{\mu}(\tau)\tilde{\vartheta}(\tau,z,\mu)v^{-\sigma}d\mu(\tau)\right)d\mu(z)  \\ &= \mathrm{CT}_{\sigma = 0}\int_{X^{(p,2)}}\left(\lim_{T\to\infty}\int_{\mathcal{F}^T}f_{\mu}(\tau)C_{00} (\tau,z,\mu)v^{-\sigma}d\mu(\tau)\right)d\mu(z)\\ &+\mathrm{CT}_{\sigma = 0}\int_{X^{(p,2)}}\left(\lim_{T\to\infty}\int_{\mathcal{F}^T}f_{\mu}(\tau)C_{0}(\tau,z,\mu)v^{-\sigma}d\mu(\tau)\right)d\mu(z)\\ &+\mathrm{CT}_{\sigma = 0}\int_{X^{(p,2)}}\left(\lim_{T\to\infty}\int_{\mathcal{F}^T}f_{\mu}(\tau)C_1(\tau,z,\mu)v^{-\sigma}d\mu(\tau)\right)d\mu(z)\\ &+\mathrm{CT}_{\sigma = 0}\int_{X^{(p,2)}}\left(\lim_{T\to\infty}\int_{\mathcal{F}^T}f_{\mu}(\tau)C_2(\tau,z,\mu)v^{-\sigma}d\mu(\tau)\right)d\mu(z).\end{align*}
Then, by employing different methods for each summand, Kudla proved that the order of integration in the four integrals can be interchanged. Specifically, the proof relies on two key facts: the convergence of $\int_{X^{(p,2)}}\tilde{\vartheta}(\tau,z,\mu)d\mu(z)$ and the holomorphy of the following functions at $\sigma = 0$, that combined, allows us to change the order of the integrals as in \eqref{eq1142}:
\begin{equation}\label{eq115}
\int_{X^{(p,2)}}\left(\lim_{T\to\infty}\int_{\mathcal{F}^T}f_{\mu}(\tau)C_{00}(\tau,z,\mu)v^{-\sigma}d\mu(\tau)\right)d\mu(z),
\end{equation}
\begin{equation}\label{eq116}
\int_{X^{(p,2)}}\left(\lim_{T\to\infty}\int_{\mathcal{F}^T}f_{\mu}(\tau)C_{0}(\tau,z,\mu)v^{-\sigma}d\mu(\tau)\right)d\mu(z),
\end{equation}
\begin{equation}\label{eq117} \int_{X^{(p,2)}}\left(\lim_{T\to\infty}\int_{\mathcal{F}^T}f_{\mu}(\tau)C_{1}(\tau,z,\mu)v^{-\sigma}d\mu(\tau)\right)d\mu(z),
\end{equation}
\begin{equation}\label{eq118}
\int_{X^{(p,2)}}\left(\lim_{T\to\infty}\int_{\mathcal{F}^T}f_{\mu}(\tau)C_{2}(\tau,z,\mu)v^{-\sigma}d\mu(\tau)\right)d\mu(z).
\end{equation}
The proof that the functions \eqref{eq115}, \eqref{eq117}, and \eqref{eq118} are holomorphic can be found in \cite[\S3, p. 24]{kudla_2003_integral} and applies to any Shimura variety associated to a quadratic space of signature $(1,2)$. Thus, by applying the truncated Siegel-Weil formula, we can interchange the order of the two integrals in \eqref{eq115}, \eqref{eq117}, and \eqref{eq118} for the modular curve setting. In contrast, the proof of the holomorphy of \eqref{eq116}, given in \cite[Proposition 3.4, p. 28]{kudla_2003_integral}, does not extend to the modular curve case. As noted earlier, we will address the integral \eqref{eq115} separately in this paper, computing it directly without interchanging the order of the integrals.
\subsubsection{Computation in the modular curve case}\label{FactorizationIntegralsSection} In the previous subsection, we outlined the approach used in \cite{kudla_2003_integral} to prepare the integral before applying the Siegel-Weil formula. As previously mentioned, the strategy involves splitting the integral into four parts, proving holomorphy in the variable $\sigma$, and then justifying the change in the order of integration. However, in the case of the modular curve, one of these four integrals fails to be holomorphic. In this subsection, we will proceed with the computation specific to the modular curve, that relies on a factorization of the integral into two parts.

In this subsection, we return to the setting of the modular curve without level. Accordingly, we fix $V$ and $L$ as the quadratic space and lattice defined in \S\ref{modexample}. Recall that $L'/L\simeq \frac{1}{2}\Z$ and $\mu_0,\mu_1$ are two distinct representatives of the cosets. We factor the theta function defined in \ref{defconvdiv} following the same idea as in the previous subsection:
\begin{align}\label{Factorizationourcase}\vartheta(\tau,z,\mu_j) =& C_{00}(\tau,z,\mu_j)+ C_0(\tau,z,\mu_j)+C_1(\tau,z,\mu_j)+C_2(\tau,z,\mu_j)\\ =& v\delta_{j0}+v\sum_{\substack{0\neq \lambda\in \mu_j+ L\\q(\lambda) = 0}}e^{-2\pi(vq(\lambda_z)-vq(\lambda_{z^{\perp}}))}\nonumber\\&+v\sum_{\substack{\lambda\in \mu_j+ L\\q(\lambda) > 0}}e^{-2\pi(vq(\lambda_z)-vq(\lambda_{z^{\perp}}))}e^{2\pi iq(\lambda)u} \nonumber\\ &+v\sum_{\substack{\lambda\in \mu_j+ L\\q(\lambda) < 0}}e^{-2\pi(vq(\lambda_z)-vq(\lambda_{z^{\perp}}))}e^{2\pi iq(\lambda)u},\nonumber
\end{align}
where $\delta_{j0} = 1$ if and only if $j = 0$ and $0$ otherwise. Since $q|_L\in \Z$ and $q(\mu_1)\not\in\Z$, it follows that $C_0(\tau,z,\mu_1) = 0$. Let $T\geq 1$, we denote \begin{align*}\mathcal{F}_1 &:= \{\tau = u+iv\in\mathcal{F}^T,\;s.t.\;|u|\leq 1/2,\;|\tau|\geq 1,\;v\leq 1\},\\
\mathcal{F}^T_2 &:= \{\tau = u+iv\in\mathcal{F}^T,\;s.t.\;|u|\leq 1/2,\;|\tau|\geq 1,\;1<v<T\},\end{align*}
so that 
\begin{equation}\label{factorizationBor}\mathcal{F}^T = \mathcal{F}_1\bigsqcup\mathcal{F}^T_2.\end{equation}
\begin{lem}\label{lema513}
The following equality holds:
\begin{align}\label{eq220}\int_{X^{mod,\hat{T}}}&\left(\int_{X^{mod}}^{\bullet}\left<f(\tau),\Theta^{Sieg}_L(\tau,z)\right>d\mu(\tau)\right)d\mu(z)  \\=& \lim_{T\to\infty}\Big[\suma\int_{\mathcal{F}^T}f_{\mu_j}(\tau)\left(\int_{X^{mod,\hat{T}}}\vartheta(\tau,z,\mu_j)d\mu(z)\right)d\mu(\tau)\nonumber\\ 
&-\int_{\mathcal{F}_2^T}f_{\mu_0}(\tau)\left(\int_{X^{mod,\hat{T}}}\vartheta(v,z,\mu_0)_0d\mu(z)\right)d\mu(\tau)\Big]\nonumber\\
&+\int_{X^{mod,\hat{T}}}\mathrm{CT}_{\sigma = 0}\lim_{T\to\infty}\int_{\mathcal{F}^T_2}f_{\mu_0}(\tau)\vartheta(v,z,\mu_0)_0v^{-\sigma}d\mu(\tau) d\mu(z).\nonumber\end{align}
\end{lem}
\begin{proof}
Using the factorization \eqref{factorizationBor} we obtain
\begin{align}
    \label{factcomp}\int_{X^{mod,\hat{T}}}&\mathrm{CT}_{\sigma = 0}\lim_{T\to\infty}\int_{\mathcal{F}^T}\left<f(\tau),\Theta^{Sieg}(\tau,z)\right>v^{1-\sigma}d\mu(\tau) d\mu(z)   \\
     =& \int_{X^{mod,\hat{T}}}\mathrm{CT}_{\sigma = 0}\lim_{T\to\infty}\int_{\mathcal{F}_2^T}\left<f(\tau),\Theta^{Sieg}(\tau,z)\right>v^{1-\sigma}d\mu(\tau) d\mu(z)\nonumber\\
    &+\int_{X^{mod,\hat{T}}}\mathrm{CT}_{\sigma = 0}\lim_{T\to\infty}\int_{\mathcal{F}_1}\left<f(\tau),\Theta^{Sieg}(\tau,z)\right>v^{1-\sigma}d\mu(\tau) d\mu(z).\nonumber
\end{align}
On the one hand, since $\mathcal{F}_1$ is compact 
\begin{align}\label{changecompact}
  \int_{X^{mod,\hat{T}}}\mathrm{CT}_{\sigma = 0}\lim_{T\to\infty}\int_{\mathcal{F}_1}&\left<f(\tau),\Theta^{Sieg}(\tau,z)\right>v^{1-\sigma}d\mu(\tau) d\mu(z)    \\
  = \mathrm{CT}_{\sigma = 0}\lim_{T\to\infty}\int_{\mathcal{F}_1}&\int_{X^{mod,\hat{T}}}\left<f(\tau),\Theta^{Sieg}(\tau,z)\right>v^{1-\sigma}d\mu(\tau) d\mu(z)\nonumber.
\end{align}
On the other hand, using the factorization given in \eqref{Factorizationourcase} we get
\begin{align}\label{eqlema2}\int_{X^{mod,\hat{T}}}&\mathrm{CT}_{\sigma = 0}\lim_{T\to\infty}\int_{\mathcal{F}^T_2}\left<f(\tau),\Theta^{Sieg}(\tau,z)\right>v^{1-\sigma}d\mu(\tau) d\mu(z)  \\
 =& \suma \int_{X^{mod,\hat{T}}}\mathrm{CT}_{\sigma = 0}\lim_{T\to\infty}\int_{\mathcal{F}^T_2}f_{\mu_j}(\tau)(C_1(\tau,z,\mu_j)+C_2(\tau,z,\mu_j))v^{-\sigma}d\mu(\tau) d\mu(z)  \nonumber\\ &+\int_{X^{mod,\hat{T}}}\mathrm{CT}_{\sigma = 0}\lim_{T\to\infty}\int_{\mathcal{F}^T_2}f_{\mu_j}(\tau)\left(C_{00}(\tau,z,\mu_j)+C_{0}(\tau,z,\mu_j)\right)v^{-\sigma}d\mu(\tau) d\mu(z).\nonumber\end{align}
As mentioned in the previous section, the proof of \cite[\S\;3,\;p.\;24]{kudla_2003_integral} applies to all orthogonal Shimura varieties. Consequently, the functions 
\begin{equation*}
    \int_{X^{mod}}\left(\lim_{T\to\infty}\int_{\mathcal{F}^T}f_{\mu_j}(\tau)C_{1}(\tau,z,\mu_j)v^{-\sigma}d\mu(\tau)\right)d\mu(z),
\end{equation*}
\begin{equation*}
    \int_{X^{mod}}\left(\lim_{T\to\infty}\int_{\mathcal{F}^T}f_{\mu_j}(\tau)C_{2}(\tau,z,\mu_j)v^{-\sigma}d\mu(\tau)\right)d\mu(z),
\end{equation*}
are holomorphic at $\sigma = 0$. Then we obtain 
\begin{align}\label{5}\int_{X^{mod,\hat{T}}}\mathrm{CT}_{\sigma = 0}\lim_{T\to\infty}\int_{\mathcal{F}_2^T}&f_{\mu_j}(\tau)(C_1(\tau,z,\mu_j)+C_2(\tau,z,\mu_j))v^{-\sigma}d\mu(\tau) d\mu(z)   \\ = \lim_{T\to\infty}\int_{\mathcal{F}_2^T}&f_{\mu_j}(\tau)\int_{X^{mod,\hat{T}}}(C_1(\tau,z,\mu_j)+C_2(\tau,z,\mu_j))d\mu(z)\mu(\tau).\nonumber\end{align}
Moreover, since $C_1(\tau,z,\mu_j)+C_2(\tau,z,\mu_j) = \vartheta(\tau,z,\mu_j)-\vartheta(\tau,z,\mu_j)_{0}$, we have 
\begin{equation}\label{eqlema1}\lim_{T\to\infty}\int_{\mathcal{F}^T_2}f_{\mu_j}(\tau)\int_{X^{mod,\hat{T}}}(C_1(\tau,z,\mu_j)+C_2(\tau,z,\mu_j))d\mu(z)\mu(\tau)  \end{equation}
\[ = \lim_{T\to\infty}\left(\int_{\mathcal{F}^T_2}f_{\mu_j}(\tau)\left(\int_{X^{mod,\hat{T}}}\vartheta(\tau,z,\mu_j)d\mu(z)\right)d\mu(\tau)- \int_{\mathcal{F}^T_2}f_{\mu_j}(\tau)\left(\int_{X^{mod,\hat{T}}}\vartheta(v,z,\mu_j)_0d\mu(z)\right)d\mu(\tau)\right),\]
where we have used the fact that the domains $\mathcal{F}^T_2$ and $X^{mod,\hat{T}}$ are compact, allowing us to express the limit of the integral as the limit of the sum of two integrals. Since by Remark \ref{130} and \cite[(3.1),\;p.\;295]{funke_2002_heegner}, we have \[\vartheta(v,z,\mu_j)_{0} = \left\{\begin{matrix}v\sum_{\substack{\lambda\in \Z^3\\q(\lambda) = 0 }}e^{-2\pi(vq(\lambda_z)-vq(\lambda_{z^{\perp}}))} &if\;j = 0\\ 0&if\;j = 1\end{matrix}\right.,\]  
plugging \eqref{eqlema1} and \eqref{5} in \eqref{eqlema2}, and then \eqref{changecompact} and \eqref{eqlema2} in \eqref{factcomp}, we obtain the result.
\end{proof}
To derive an explicit expression for \eqref{eq220}, we will analyze each factor on the right-hand side of the previous proposition separately.
\begin{defn}\label{limit}
\em
The integral 
\begin{equation*}\int_{X^{mod,\hat{T}}}\mathrm{CT}_{\sigma = 0}\lim_{T\to\infty}\int_{\mathcal{F}_2^T}f_{\mu_0}(\tau)\vartheta(v,z,\mu_0)_0v^{-\sigma}d\mu(\tau) d\mu(z),\end{equation*}
is called the \textit{limit case}. We refer to the integrals
$$\lim_{T\to\infty}\left(\int_{\mathcal{F}^T}f_{\mu_j}(\tau)\left(\int_{X^{mod,\hat{T}}}\vartheta(\tau,z,\mu_j)d\mu(z)\right)d\mu(\tau)- \int_{\mathcal{F}_2^T}f_{\mu_0}(\tau)\left(\int_{X^{mod,\hat{T}}}\vartheta(v,z,\mu_0)_0d\mu(z)\right)d\mu(\tau)\right),$$
and 
\[\lim_{T\to\infty}\int_{\mathcal{F}^T}f_{\mu_1}(\tau)\left(\int_{X^{mod,\hat{T}}}\vartheta(\tau,z,\mu_1)d\mu(z)\right)d\mu(\tau),\]
as the \textit{ordinary cases}.
\em
\end{defn}
\subsection{Ordinary cases}\label{216}
In this subsection, we compute the ordinary cases using the geometric version of the truncated Siegel-Weil formula developed earlier. Additionally, we will rely on auxiliary computations established in \S\ref{231}.\\

With the aim of simplifying the argument, we divide the computation into two terms:
$$\int_{\mathcal{F}^T}f_{\mu_j}(\tau)\left(\int_{X^{mod,\hat{T}}}\vartheta(\tau,z,\mu_j)d\mu(z)\right)d\mu(\tau),\;\; \int_{\mathcal{F}_2^T}f_{\mu_0}(\tau)\left(\int_{X^{mod,\hat{T}}}\vartheta(v,z,\mu_0)_0d\mu(z)\right)d\mu(\tau).$$
\begin{lem}\label{174}
The convergent part and its Fourier constant term, as defined in \ref{defconvdiv}, satisfy the following estimates:
\begin{align*}\int_{\mathcal{F}_1}f_{\mu_j}(\tau)\int_{\widehat{X^{mod,\hat{T}}}}\mathbf{Conv}(\tau,z,\mu_j)d\mu(z)d\mu(\tau) &= O(e^{-\hat{T}}).\\ \int_{\mathcal{F}_1}f_{\mu_0}(\tau)\int_{\widehat{X^{mod,\hat{T}}}}\mathbf{Conv}(v,z,\mu_0)_0d\mu(z)d\mu(\tau) &= O(e^{-\hat{T}}).\end{align*}
\end{lem}
\begin{proof}
By Lemma \ref{rapdec} the function  $\mathbf{Conv}(\tau,z,\mu_j)$ is exponentially decreasing in the variable $y$ of $z = x+iy$, then
\[|\mathbf{Conv}(\tau,z,\mu_j)|\leq |F_{\mu_j}(\tau)|e^{-y},\]
where $F_{\mu_j}(\tau)$ is a continuous function. Hence
\[\left|\int_{\mathcal{F}_1}f_{\mu_j}(\tau)\int_{\widehat{X^{mod,\hat{T}}}}\mathbf{Conv}(\tau,z,\mu_j)d\mu(z)d\mu(\tau)\right| \leq e^{-\hat{T}}\int_{\mathcal{F}_1}\left|f_{\mu_j}(\tau)F_{\mu_j}(\tau)\right|d\mu(\tau).\]
Using the compactness of $\mathcal{F}_1$ we obtain the estimate of the statement. The proof for $\mathbf{Conv}(v,z,\mu_0)_0$ is completely analogous.
\end{proof}
\begin{prop}\label{220}
The following equality holds:
\begin{align*}\suma \int_{\mathcal{F}^T}f_{\mu_j}(\tau)\left(\int_{X^{mod,\hat{T}}}\vartheta(\tau,z,\mu_j)d\mu(z)\right)d\mu(\tau)  &= 2\vol(X^{mod})\sum_{j = 0}^1\sum_{m\in j/2 +\Z^{\geq0}} c_{\mu_j}(-m)b(m,T,\varphi_{\mu_j})\\&+O(T^{-1/2}),\end{align*}
where we recall that $b(m,T,\varphi_{\mu_j})$ is the first term at $s = 1/2$ of the Laurent series of the $m-$Fourier coefficient of $E(\tau,s,3/2,\mu(\varphi_{\mu_j}))$ evaluated at $T$, see \eqref{eqLaurent}. 
\end{prop}
\begin{proof}
We recall that Theorem \ref{164} provides 
\begin{align*}\int_{X^{mod,\hat{T}}}&\vartheta(\tau,z,\mu(\varphi_{\mu_j}))d\mu(z) = \vol(X^{mod})\Big[\CTE(\tau,-1/2,\mu(\varphi_{\mu_j}))\\&+cA_{-1}(\tau,-1/2,\mu(\tilde{\varphi}_j))\Big]+\int_{\widehat{X^{mod,\hat{T}}}}\mathbf{Conv}(\tau,z,\mu_j)d\mu(z)\\&+\int_{X^{mod,\hat{T}}}\mathbf{Div}(\tau,z,\mu_j)d\mu(z).\end{align*}
Here, we recall that $\tilde{\varphi}_j\in \mathcal{S}(V(\A))$ is an unspecified Schwartz function obtained in the Siegel-Weil formula of Theorem \ref{swcor}. Then, using the above equality we factor the integral of the statement into the sum of three terms:
\begin{align}\label{eq237}\vol(X^{mod})&\left[\int_{\mathcal{F}^T}f_{\mu_j}(\tau)\CTE(\tau,-1/2,\mu(\varphi_{\mu_j}))d\mu(\tau)+\int_{\mathcal{F}^T}f_{\mu_j}(\tau)\RESE(\tau,-1/2,\mu(\tilde{\varphi}_j))d\mu(\tau)\right] \\\label{eq2731}&-\int_{\mathcal{F}^T}f_{\mu_j}(\tau)\int_{\widehat{X^{mod,\hat{T}}}}\mathbf{Conv}(\tau,z,\mu_j)d\mu(z)d\mu(\tau)\\ \label{eq2732}&+\int_{\mathcal{F}^T}f_{\mu_j}(\tau)\int_{X^{mod,\hat{T}}}\mathbf{Div}(\tau,z,\mu_j)d\mu(z)d\mu(\tau).\end{align}
We consider each of the above factors separately. Using Proposition \ref{230} and Lemma \ref{lemholom} in \eqref{eq237}, Lemma \ref{174} in \eqref{eq2731} and Proposition \ref{226} in \eqref{eq2732} we obtain the result.
\end{proof}
\begin{prop}\label{217}
We obtain  \begin{align*}\int_{\mathcal{F}^T_2}&f_{\mu_0}(\tau)\left(\int_{X^{mod,\hat{T}}}\vartheta(\tau,z,\mu_0)_0d\mu(z)\right)d\mu(\tau) \\&=c_{\mu_0}(0)\mathrm{vol}(X^{mod})\left(\log(T)-2\log(\hat{T}) + 2\left(-\tanh^{-1}\left(\frac{\sqrt{7}}{4}\right)+\frac{\sqrt{7}}{4}+ \frac{\log(3/4)}{2}\right)\right)+O(T^{-1/4}).\end{align*}
\end{prop}
\begin{proof}
We recall that Corollary \ref{237} shows 
\begin{align*}\int_{X^{mod,\hat{T}}}&\vartheta(v,z,\mu_0)_0d\mu(z) =  \vol(X^{mod})\Big[\CTE(v,-1/2,\mu(\varphi_{\mu_0}))_0\\&+c\RESE(v,-1/2,\mu(\tilde{\varphi}_0))_0\Big]- \int_{\widehat{X^{mod,\hat{T}}}}\mathbf{Conv}(v,z,\mu_0)_0d\mu(z)\\&+\int_{X^{mod,\hat{T}}}\mathbf{Div}(v,z,\mu_0)_0d\mu(z).\end{align*}
We proceed as in the proof of Proposition \ref{220}. Using the above equality on the integral of the statement, we find
\begin{align}\label{eq241}\int_{\mathcal{F}^T_2}f_{\mu_0}(\tau)\left(\int_{X^{mod,\hat{T}}}\vartheta(v,z,\mu_0)_0d\mu(z)\right)d\mu(\tau) =& \vol(X^{mod})\Big[\int_{\mathcal{F}^T}f_{\mu_0}(\tau)\CTE(v,-1/2,\mu(\varphi_{\mu_0}))_0d\mu(\tau)\\\label{zero2}&+\int_{\mathcal{F}^T_2}f_{\mu_0}(\tau)\RESE(v,-1/2,\mu(\tilde{\varphi}_0))_0d\mu(\tau)\Big]\\\label{zero3}&+\int_{\mathcal{F}^T_2}f_{\mu_0}(\tau)\int_{X^{mod,\hat{T}}}\mathbf{Div}(v,z,\mu_0)_0d\mu(z)d\mu(\tau)\\\label{zero4} &-\int_{\mathcal{F}^T_2}f_{\mu_0}(\tau)\int_{\widehat{X^{mod,\hat{T}}}}\mathbf{Conv}(v,z,\mu_0)_0d\mu(z)d\mu(\tau).\end{align}
As we did in the proof of Proposition \ref{220} we consider each factor of the above equation separately. By Corollary \ref{corholo} we show that \eqref{zero2} vanishes. Moreover using Corollary \ref{238} on the right hand side of \eqref{eq241}, Proposition \ref{232} in \eqref{zero3} and Lemma \ref{174} in \eqref{zero4}, the statement follows.
\end{proof}
\begin{cor}\label{218}
The following equality holds:
\begin{align*}\lim_{T\to\infty}&\left(\suma\int_{\mathcal{F}^T}f_{\mu_j}(\tau)\left(\int_{X^{mod,\hat{T}}}\vartheta(\tau,z,\mu_j)d\mu(z)\right)d\mu(\tau)-\int_{\mathcal{F}^T}f_{\mu_0}(\tau)\left(\int_{X^{mod,\hat{T}}}\vartheta(v,z,\mu_0)_0d\mu(z)\right)d\mu(\tau)\right)  \\ =& 2\vol(X^{mod})\sum_{j = 0}^1\sum_{m\in j/2+\Z^{> 0}} c_{\mu_j}(-m)b(m,\varphi_{\mu_j})\\&+2c_{\mu_0}(0)\mathrm{vol}(X^{mod})\left(-\tanh^{-1}\left(\frac{\sqrt{7}}{4}\right)+\frac{\sqrt{7}}{4}+ \frac{\log(3/4)}{2}\right)-2c_{\mu_0}(0)\mathrm{vol}(X^{mod})\log(\hat{T}),\end{align*}
where $b(m,\varphi_{\mu_j}) = \lim_{T\to\infty}b(m,T,\varphi_{\mu_j})$.
\end{cor}
\begin{proof}
By \cite[Theorem\;6.6,\;p.\;2303]{kudla_2010_eisenstein}, $\lim_{T\to\infty}\left(2b(0,T,\varphi_{\mu_0})-\log(T)\right) = 0$. Moreover, \cite[Theorem 6.6, p. 2303]{kudla_2010_eisenstein} shows that $b(0,T,\varphi_{\mu_1}) = 0$. Then, using Propositions \ref{220} and \ref{217}, we conclude the proof.
\end{proof}
\subsection{Limit case}\label{219}
This section is devoted to compute the limit case of Definition \ref{limit}:
\begin{equation}\label{limit1}\int_{X^{mod,\hat{T}}}\left(\mathrm{CT}_{\sigma = 0}\lim_{T\to\infty}\int_{\mathcal{F}_2^T}f_{\mu_0}(\tau)\vartheta(v,z,\mu_0)_0v^{-\sigma}d\mu(\tau)\right)d\mu(z).\end{equation}
The main tool for this computation is the Rankin-Selberg method for truncated fundamental domains, as developed in \cite{Don1982TheRM}. We outline the strategy used in this section as follows: In Proposition \ref{227}, we introduce an auxiliary variable $y^s$ to unfold the integral with the constant term of the theta function, resulting in a sum of three unfolded integrals. We then analyze each integral separately in Lemmas \ref{derivadaeis}, \ref{212} and \ref{221}, relating them to the truncated Rankin-Selberg integral from \cite{Don1982TheRM}. By applying the main result of \textit{loc.cit.}, we evaluate each integral. Furthermore, in Corollary \ref{213}, we combine the previous computations to obtain the result for the limit case.\\

Before proceeding with the computation, we introduce some concepts used in the truncated Rankin-Selberg method. For every $\hat{T}\geq 1$, the subset $X^{mod,\hat{T}} = \{z = x+iy\in \mathcal{H},\;s.t.\;|z|\geq 1,\;|x|\leq 1/2,\;y\leq \hat{T}\}$ is a fundamental domain for the action of $\mathrm{SL}_2(\Z)$ on 
$$\mathcal{H}^{\hat{T}} := \bigcup_{\gamma\in SL_2(\Z)}\gamma X^{mod,\hat{T}} = \{z\in \mathcal{H},\;s.t.\;\mathrm{max}_{\gamma\in SL_2(\Z)}\mathrm{Im}(\gamma z)\leq \hat{T}\}.$$
According to \cite[(20),\;p.\;420]{Don1982TheRM}
$$\mathcal{H}^{\hat{T}} = \{z\in\mathcal{H},\;s.t.\;\mathrm{Im}(z)\leq \hat{T}\}-\bigcup_{c\geq 1}\bigcup_{\substack{a\in\Z\\(a,c) = 1}}S_{a/c},$$
where $S_{a/c}$ is the disc of radius $\frac{1}{2c^2\hat{T}}$ tangent to the real axis at $a/c$. Therefore, given $$\Gamma^{\infty} = \left\{\begin{pmatrix}1&x\\0&1\end{pmatrix},\;x\in\Z\right\},$$
we define the following subset of the complex numbers:
\begin{equation}\label{truncado} \Gamma^{\infty}\setminus \mathcal{H}^{\hat{T}} := \{z = x+iy\in\mathcal{H},\;|x|\leq 1/2,\;0\leq y\leq \hat{T}\}-\bigcup_{c\geq 1}\bigcup_{\substack{a\in\Z\\(a,c) = 1}}S_{a/c}.\end{equation}
\begin{defn}\label{defeis2}
\em
Let $s\in\C$ such that $\mathrm{Re}(s)>1$. The \textit{classical Eisenstein series} considered by Zagier is defined as follows:
\begin{align*}E(z,s) = \sum_{\gamma\in\Gamma^{\infty}\setminus \mathrm{SL}_2(\Z)}\mathrm{Im}(\gamma z)^s.\end{align*}
\em
\end{defn}
\begin{prop}
The function $E(z,s)$ is holomorphic in $\mathrm{Re}(s)>1/2$ except for a pole of residue $3/\pi$ at $s = 1$. Furthermore $\zeta^*(2s)E(z,s)$ is holomorphic in all $s\in\C$ except for $s\neq 0,1$, where $\zeta^*(s) := \pi^{-s/2}\Gamma(s/2)\zeta(s)$ is the completed Riemann zeta function.
\end{prop}
\begin{proof}
We refer the reader to \cite[p.\;415,\;p.\;416]{Don1982TheRM}.
\end{proof}
\begin{obs}
The $0-$th Fourier term of the modular form $E(z,s)$ is equal to 
\[y^s+\varphi(s)y^{1-s},\]
where $\varphi(s) := \frac{\zeta^*(2s-1)}{\zeta^*(2s)}$.
\end{obs}
\begin{prop}\label{zagier}
The following equality holds:
\[\int_{X^{mod,\hat{T}}}E(z,s)\frac{dxdy}{y^2} = \hat{T}^{s-1}/(s-1)-\frac{\zeta^*(2s-1)}{\zeta^*(2s)}\hat{T}^{-s}/s,\]
where $\zeta^*(s) = \pi^{-s/2}\Gamma(s/2)\zeta(s)$ is the completed Riemann zeta function.
\end{prop}
\begin{proof}
See \cite[(33),\;p.\;426]{Don1982TheRM}.
\end{proof}
In the following Proposition, we will unfold our integral and divide the main computation into smaller parts, which will subsequently be addressed using the truncated Rankin-Selberg method from \cite{Don1982TheRM}.
\begin{prop}\label{227}
Let $A$ be the constant defined in \S\ref{notation}. The limit case satisfies the following equality:
\begin{align*}\int_{X^{mod,\hat{T}}}&\left(\mathrm{CT}_{\sigma = 0}\lim_{T\to\infty}\int_{\mathcal{F}^T_2}f_{\mu_0}(\tau)\vartheta(v,z,\mu_0)_0v^{-\sigma}d\mu(\tau)\right)d\mu(z)  \\=& c_{\mu_0}(0)\Bigg(A\mathrm{CT}_{s = 0}\int_{\Gamma^{\infty}\setminus \mathcal{H}^{\hat{T}}}y^s\frac{dxdy}{y^2} - 8\erf\left(\sqrt{\frac{\pi}{2}}\right)\mathrm{CT}_{s = 0}\int_{\Gamma^{\infty}\setminus \mathcal{H}^{\hat{T}}}y^s\log(y)\frac{dxdy}{y^2}\\ &+\mathrm{CT}_{s = 0}\int_{\Gamma^{\infty}\setminus \mathcal{H}^{\hat{T}}}y^{s}\frac{dydx}{y}\Bigg),\end{align*}
\end{prop}
\begin{proof}
We factor $\vartheta(v,z,\mu_0)_0 = C_0(v,z,\mu_0)+C_{00}(v,z,\mu_0)$, where we recall the notation
$C_0(v,z,\mu_0) = v\sum_{\substack{0\neq \lambda\in \Z^3\\q(\lambda) = 0}}e^{-4\pi v q(\lambda_z)},$ and $C_{00}(v,z,\mu_0) = v$. We obtain
\begin{align*}\int_{X^{mod,\hat{T}}}&\left(\mathrm{CT}_{\sigma = 0}\lim_{T\to\infty}\int_{\mathcal{F}^T_2}f_{\mu_0}(\tau)\vartheta(v,z,\mu_0)_0v^{-\sigma}d\mu(\tau)\right)d\mu(z) \\
 =& \int_{X^{mod,\hat{T}}}\left(\mathrm{CT}_{\sigma = 0}\lim_{T\to\infty}\int_{\mathcal{F}^T_2}f_{\mu_0}(\tau)C_0(v,z,\mu_0)v^{-\sigma}d\mu(\tau)\right)d\mu(z)\\ &+\int_{X^{mod,\hat{T}}}\left(\mathrm{CT}_{\sigma = 0}\lim_{T\to\infty}\int_{\mathcal{F}^T_2}f_{\mu_0}(\tau)C_{00}(v,z,\mu_0)v^{-\sigma}d\mu(\tau)\right)d\mu(z).\end{align*}
The second integral on the right hand side is equal to
\begin{align*}\mathrm{CT}_{\sigma = 0}\lim_{T\to\infty}\int_{\mathcal{F}^T_2}f_{\mu_0}(\tau)C_{00}(v,z,\mu_0)v^{-\sigma}d\mu(\tau)= \mathrm{CT}_{\sigma = 0}\lim_{T\to\infty}\int_{1}^Tc_{\mu_0}(0)v^{-\sigma-1}dv = 0,\end{align*}
Therefore, by direct computation we get
\begin{align}\int_{X^{mod,\hat{T}}}&\left(\mathrm{CT}_{\sigma = 0}\lim_{T\to\infty}\int_{\mathcal{F}^T_2}f_{\mu_0}(\tau)\vartheta(v,z,\mu_0)_0v^{-\sigma}d\mu(\tau)\right)d\mu(z)   \nonumber \\
\label{eq202} &= c_{\mu_0}(0)\int_{X^{mod,\hat{T}}}\left(\mathrm{CT}_{\sigma = 0}\lim_{T\to\infty}\int_{1}^T\sum_{\substack{0\neq \lambda\in \Z^3\\q(\lambda) = 0}}e^{-4\pi v q(\lambda_z)}v^{-\sigma-1}dv\right)d\mu(z).\end{align}
The group $SO(V)(\R)$ acts on $V(\R)$ by the standard representation, while $\mathrm{SL}_2(\R)$ acts on $V$ via conjugation on $V$, represented in the trace-zero matrix form discussed in \S\ref{modexample}. These actions define a $2$-to-$1$ morphism $\mathrm{SL}_2(\R)\to SO(V)(\R)$. Under this correspondence, the non-zero isotropic elements of the quadratic space, i.e. $0\neq \lambda\in V$ such that $q(\lambda) = 0$, are generated by $\mathrm{SL}_2(\Z)$ in one orbit. The isotropic vectors are in one to one correspondence with $\Q\cup\{\infty\}$, hence $\mathrm{SL}_2(\Z)\setminus\{0\neq \lambda\in V(\Q)\}$ is identified with the cusp $\infty$ of $X^{mod}$. One representative of the cusp $\infty$ in the projective cone model, described in Proposition \ref{geom20}, is the isotropic line $\Q^{\times}\cdot f_1$. Its stabilizer in $\mathrm{SL}_2(\Z)$ is given by the unipotent subgroup
$$\Gamma_{\infty} := \left\{\begin{pmatrix}1&x\\0&1\end{pmatrix},\;x\in\Z\right\}.$$
Therefore we obtain the following identification
\begin{equation}\label{isomunfold}\{0\neq \lambda\in V,\;s.t.\;q(\lambda) = 0\} \simeq  \Gamma_{\infty}\setminus \mathrm{SL}_2(\Z)\cdot(\Q^{\times}\cdot f_1),\end{equation}
and hence rewrite the inner sum of \eqref{eq202} as follows:
\begin{align*}\sum_{\substack{0\neq \lambda\in \Z^3\\q(\lambda) = 0}}e^{-4\pi vq(\lambda_z)} &= \sum_{\gamma\in\Gamma_{\infty}\setminus \mathrm{SL}_2(\Z)}\;\sum_{x_2\in\Z\setminus \{0\}}e^{-4\pi v q\left(\gamma\cdot(x_2,0,0)_z\right)}.
\end{align*}
Here the action of $\mathrm{SL}_2(\R)$ on $V$ is via conjugation, as we mentioned before. By the invariance property of the Gaussian we obtain
$$\sum_{\gamma\in\Gamma_{\infty}\setminus \mathrm{SL}_2(\Z)}\;\sum_{x_2\in\Z\setminus \{0\}}e^{-4\pi v q(\gamma\cdot(x_2,0,0)_z)} = \sum_{\gamma\in\Gamma_{\infty}\setminus \mathrm{SL}_2(\Z)}\;\sum_{x_2\in\Z\setminus \{0\}}e^{-4\pi v q\left((x_2,0,0)_{\gamma^{-1}z}\right)}.$$
We would like to unfold the integral over $X^{mod,\hat{T}}$. In order to overcome the convergence problems of the unfolding we introduce the auxiliary term $\mathrm{Im}(z)^s$ with $s\in\C$. Consequently, the integral of  \eqref{eq202} is equal to 
\begin{align}\int_{X^{mod,\hat{T}}}&\left(\mathrm{CT}_{\sigma = 0}\lim_{T\to\infty}\int_{1}^T\sum_{\gamma\in\Gamma_{\infty}\setminus \mathrm{SL}_2(\Z)}\;\sum_{x_2\in\Z\setminus \{0\}}e^{-4\pi v q((x_2,0,0)_{\gamma^{-1}z})}v^{-\sigma-1}\mathrm{CT}_{s = 0}\mathrm{Im}(\gamma^{-1} z)^sdv\right)d\mu(z) \nonumber\\ \label{eq211}=& \mathrm{CT}_{s = 0}\int_{X^{mod,\hat{T}}}\left(\mathrm{CT}_{\sigma = 0}\lim_{T\to\infty}\int_{1}^T\sum_{\gamma\in\Gamma_{\infty}\setminus \mathrm{SL}_2(\Z)}\;\sum_{x_2\in\Z\setminus \{0\}}e^{-4\pi v q((x_2,0,0)_{\gamma^{-1}z})}v^{-\sigma-1}\mathrm{Im}(\gamma^{-1} z)^sdv\right)d\mu(z),\end{align}
where the equality is justified by means of Fubini's Theorem. To unfold the integral \eqref{eq211} with the sum, we suppose that $s\in\C$ satisfies that $\mathrm{Re}(s)>1$ to ensure the convergence of the resulting function. After the unfolding, by means of meromorphic continuation, the result will follow. Under this assumption on $s$, we can unfold the integral over $X^{mod, \hat{T}}$ with the sum over $\gamma\in \Gamma_{\infty}\setminus \mathrm{SL}_2(\Z)$, obtaining that \eqref{eq211} is equal to
\begin{equation}\label{eq225}\mathrm{CT}_{s = 0}\int_{\Gamma^{\infty}\setminus X^{mod,\hat{T}}}\left(\mathrm{CT}_{\sigma = 0}\lim_{T\to\infty}\int_{1}^T\sum_{x_2\in\Z\setminus \{0\}}e^{-4\pi vq((x_2,0,0)_{z})}v^{-\sigma-1}dv\right)y^sd\mu(z).\end{equation}
Using the formula given in \cite[(3.9),\;p.\;296]{funke_2002_heegner} it holds that
$$q((x_2,0,0)_z) = \frac{x_2^2}{y^2}.$$
By Poisson summation formula we obtain 
\begin{equation}\label{eq223}\sum_{x_2\in\Z\setminus \{0\}}e^{ \frac{-4\pi v x_2^2}{y^2}} = \sum_{w_1\in\Z}\int_{\R}e^{\frac{-4\pi v x_2^2}{y^2}+2\pi i x_2w_1}dx_2-1,\end{equation}
where the $-1$ term in the above equation corresponds to the term $e^{ \frac{-2\pi v x_2^2}{y^2}}$ evaluated at $x_2 = 0$.
Let us divide $\Z = \Z\setminus \{0\}\cup \{0\}$, we factor the right hand side of \eqref{eq223} as follows:
\begin{equation}\label{eq224}\sum_{w_1\in\Z}\int_{\R}e^{\frac{-4\pi v x_2^2}{y^2}+2\pi i x_2w_1}dx_2-1 = \sum_{ w_1\in\Z\setminus \{0\}}\int_{\R}e^{\frac{-4\pi v x_2^2}{y^2}+2\pi i x_2w_1}dx_2+ \int_{\R}e^{\frac{-4\pi v x_2^2}{y^2}}dx_2-1.\end{equation}
Plugging the factorization \eqref{eq224} into the integral \eqref{eq225} we obtain
\begin{align}\label{eq1aux}\mathrm{CT}_{s = 0}&\int_{\Gamma^{\infty}\setminus X^{mod,\hat{T}}}\left(\mathrm{CT}_{\sigma = 0}\lim_{T\to\infty}\int_{1}^T\sum_{x_2\in\Z\setminus \{0\}}e^{-4\pi vq((x_2,0,0)_{z})}v^{-\sigma-1}dv\right)y^sd\mu(z)   \\
\label{eq2461}
=& \mathrm{CT}_{s = 0}\int_{\Gamma^{\infty}\setminus X^{mod,\hat{T}}}\left(\mathrm{CT}_{\sigma = 0}\lim_{T\to\infty}\int_{1}^T\int_{\R}e^{\frac{-4\pi v x_2^2}{y^2}}dx_2v^{-\sigma-1}dv\right)y^sd\mu(z)  \\
\label{eq2462}&+\mathrm{CT}_{s = 0}\int_{\Gamma^{\infty}\setminus X^{mod,\hat{T}}}\left(\mathrm{CT}_{\sigma = 0}\lim_{T\to\infty}\int_{1}^T\sum_{w_1\in\Z\setminus \{0\}}\int_{\R}e^{\frac{-4\pi v x_2^2}{y^2}+2\pi i x_2w_1}dx_2v^{-\sigma-1}dv\right)y^sd\mu(z)\\
\label{eq2463}&+\mathrm{CT}_{s = 0}\int_{\Gamma^{\infty}\setminus \mathcal{H}^{\hat{T}}}\left(\mathrm{CT}_{\sigma = 0}\lim_{T\to\infty}\int_{1}^Tv^{-\sigma-1}dv\right)y^sd\mu(z).\end{align}
The goal of this Proposition is achieved by computing the sum of the above three integrals. Applying Lemma \ref{224} and the proof of Lemma \ref{221} to \eqref{eq2461} we obtain that
\[\int_{\Gamma^{\infty}\setminus \mathcal{H}^{\hat{T}}}\left(\mathrm{CT}_{\sigma = 0}\lim_{T\to\infty}\int_{1}^T\int_{\R}e^{\frac{-4\pi v x_2^2}{y^2}}dx_2v^{-\sigma-1}dv\right)y^sd\mu(z),\]
is a meromorphic function in the variable $s\in\C$. We apply
Lemma \ref{225} and the proof of Lemmas \ref{212} and \ref{derivadaeis} to \eqref{eq2462} to obtain that
\[\int_{\Gamma^{\infty}\setminus \mathcal{H}^{\hat{T}}}\left(\mathrm{CT}_{\sigma = 0}\lim_{T\to\infty}\int_{1}^T\sum_{w_1\in\Z\setminus \{0\}}\int_{\R}e^{\frac{-4\pi v x_2^2}{y^2}+2\pi i x_2w_1}dx_2v^{-\sigma-1}dv\right)y^sd\mu(z),\]
is a meromorphic function in the variable $s\in\C$.
Applying Lemma \ref{243} to \eqref{eq2463}, the function
\[\int_{\Gamma^{\infty}\setminus \mathcal{H}^{\hat{T}}}\left(\mathrm{CT}_{\sigma = 0}\lim_{T\to\infty}\int_{1}^Tv^{-\sigma-1}dv\right)y^sd\mu(z),\]
is meromorphic in the variable $s\in\C$. Hence by meromorphic continuation in $s\in\C$ we obtain that the equality \eqref{eq1aux} holds for every $s\in \C$. To conclude we use Lemma \ref{221} in  \eqref{eq2461}, Lemma \ref{225} in \eqref{eq2462} and Lemma \ref{243} in \eqref{eq2463}.
\end{proof}
\begin{lem}\label{224}
Let $s\in\C$ so that $\mathrm{Re}(s)>1$. The integral \eqref{eq2461} satisfies
\begin{align*}\int_{\Gamma^{\infty}\setminus \mathcal{H}^{\hat{T}}}\left(\mathrm{CT}_{\sigma = 0}\lim_{T\to\infty}\int_{1}^T\int_{\R}e^{\frac{-4\pi v x_2^2}{y^2}}dx_2v^{-\sigma-1}dv\right)y^sd\mu(z) = \int_{\Gamma^{\infty}\setminus \mathcal{H}^{\hat{T}}}y^{s}\frac{dydx}{y}.\end{align*}
\end{lem}
\begin{proof}
By direct computation 
$$\int_{\R}e^{\frac{-4\pi v x_2^2}{y^2}}dx_2 = \frac{y}{2v^{1/2}}.$$
Therefore 
\begin{equation}\label{eq242}\mathrm{CT}_{\sigma = 0}\lim_{T\to\infty}\int_{1}^T\int_{\R}e^{\frac{-4\pi v x_2^2}{y^2}}dx_2v^{-\sigma-1}dv = \frac{y}{2}\mathrm{CT}_{\sigma = 0}\lim_{T\to\infty}\int_{1}^Tv^{-\sigma-3/2}dv.\end{equation}
Let us suppose that $\mathrm{Re}(\sigma )>-1/2$, then 
\[\lim_{T\to\infty}\int_{1}^Tv^{-\sigma-3/2} = \frac{2}{2\sigma+1}.\]
By meromorphic continuation 
\[\frac{y}{2}\mathrm{CT}_{\sigma = 0}\lim_{T\to\infty}\int_{1}^Tv^{-\sigma-3/2}dv  = y.\]
Plugging the equality \eqref{eq242} into the integral of the statement the result follows.
\end{proof}
\begin{lem}\label{225}
Let $A$ be the constant defined in \ref{notation}. Given $s\in\C$ such that $\mathrm{Re}(s)>1$, the integral \eqref{eq2462} satisfies the following equality:
\begin{align*}\int_{\Gamma^{\infty}\setminus \mathcal{H}^{\hat{T}}}&\left(\mathrm{CT}_{\sigma = 0}\lim_{T\to\infty}\int_{1}^T\sum_{w_1\in\Z\setminus \{0\}}\int_{\R}e^{\frac{-4\pi v x_2^2}{y^2}+2\pi i x_2w_1}dx_2v^{-\sigma-1}dv\right)y^sd\mu(z) \\&= A\int_{\Gamma^{\infty}\setminus \mathcal{H}^{\hat{T}}}y^s\frac{dxdy}{y^2}  -2\erf\left(\sqrt{\frac{\pi}{2}}\right)\int_{\Gamma^{\infty}\setminus \mathcal{H}^{\hat{T}}}y^s\log(y)\frac{dxdy}{y^2}.\end{align*}

\end{lem}
\begin{proof}
By direct computation 
$$\int_{\R}e^{\frac{-4\pi v x_2^2}{y^2}+2\pi i x_2w_1}dx_2 = \frac{y}{2\sqrt{v}}e^{\frac{-\pi w_1^2y^2}{4v}}.$$
Then, we obtain
\begin{align}\mathrm{CT}_{\sigma = 0}&\lim_{T\to\infty}\int_{1}^T\sum_{w_1\in\Z\setminus \{0\}}\int_{\R}e^{\frac{-4\pi v x_2^2}{y^2}+2\pi i x_2w_1}dx_2v^{-\sigma-1}dv  \\ &= \frac{y}{2}\left(\mathrm{CT}_{\sigma = 0}\lim_{T\to\infty}\int_{1}^T\sum_{w_1\in\Z\setminus \{0\}}e^{\frac{-\pi w_1^2y^2}{4v}}v^{-\sigma-3/2}dv\right).\nonumber\end{align}
We make a change of variables of the form $2vw_1^{-2} = v$, obtaining  
\begin{align}\label{eq243}\frac{y}{2}\mathrm{CT}_{\sigma = 0}&\lim_{T\to\infty}\int_{1}^T\sum_{w_1\in\Z\setminus \{0\}}e^{\frac{-\pi w_1^2y^2}{4v}}v^{-\sigma-3/2}dv \\ \label{eq2432}  &=\mathrm{CT}_{\sigma = 0}\left(2^{\sigma+1/2}\left(\sum_{w_1\in\Z\setminus \{0\}}w_1^{-2\sigma-1}\right)\left(y\lim_{T\to\infty}\int_{1}^Te^{\frac{-\pi y^2}{2v}}v^{-\sigma-3/2}dv\right)\right).\end{align}
In order to obtain an explicit formula for \eqref{eq243} we have to write the Laurent series of each factor of \eqref{eq2432}. First we consider the integral in \eqref{eq2432}. Making a change of variables of the form $\frac{y^2}{v} = \frac{1}{v}$, it holds 
\begin{equation}\label{auxlemm}y\lim_{T\to\infty}\int_{1}^Te^{\frac{-\pi y^2}{2v}}v^{-\sigma-3/2}dv = y^{-2\sigma}\lim_{T\to\infty}\int_1^Te^{\frac{-\pi}{2v}}v^{-\sigma-3/2}dv.\end{equation}
The above function is holomorphic at $\sigma = 0$, then we have to consider the constant and first term of the Laurent expansion of \eqref{auxlemm}. We rewrite \eqref{auxlemm} as follows:
\begin{equation}\label{eqlemm2}\lim_{T\to\infty}\int_{1}^Te^{\frac{-\pi }{2v}}v^{-\sigma-3/2}dv = \left(\frac{\pi}{2}\right)^{-\sigma+1/2}\Gamma\left(-1/2+\sigma,\frac{\pi}{2}\right),\end{equation}
where $\Gamma(\cdot,\cdot)$ is the incomplete gamma function. The constant term is equal to  
\begin{equation}\label{eq245}\lim_{T\to\infty}\int_1^Te^{\frac{-\pi}{2v}}v^{-3/2}dv = \sqrt{2}\erf\left(\sqrt{\frac{\pi}{2}}\right).\end{equation}
Furthermore the first term of the Laurent expansion of \eqref{eqlemm2} satisfies
\begin{align}\label{firstterm}\mathrm{FT}_{\sigma = 0}\lim_{T\to\infty}\int_{1}^Te^{\frac{-\pi }{2v}}v^{-\sigma-3/2}dv &=  \left(\frac{\pi}{2}\right)^{1/2}\left(\Gamma'\left(-1/2,\frac{\pi}{2}\right)- \Gamma\left(-1/2,\frac{\pi}{2}\right)\log\left(\frac{\pi}{2}\right)\right) \\ &=: \tilde{B},\nonumber\end{align}
where $\Gamma(a,b)$ is the incomplete Gamma function.
Using \eqref{firstterm}, we obtain  
\begin{equation}\label{first1}\mathrm{FT}_{\sigma = 0}y^{-2\sigma}\lim_{T\to\infty}\int_1^Te^{\frac{-\pi}{2v}}v^{-\sigma-3/2}dv  =\tilde{B}-2\log(y)\sqrt{2}\erf\left(\sqrt{\frac{\pi}{2}}\right).\end{equation}
Now we consider the sum of the equation \eqref{eq243}. We may observe that   
$$\sum_{w_1\in\Z\setminus \{0\}}w_1^{-2\sigma-1} = \zeta(2\sigma+1)-1.$$
The following equalities are well known
\begin{align}\label{eq244}\mathrm{CT}_{\sigma = 0}\zeta(2\sigma+1) &= -\gamma \\
\label{reszeta}\mathrm{Res}_{\sigma = 0}\zeta(2\sigma+1) &= \frac{1}{2}.\end{align}
Furthermore \begin{equation}\label{el2}2^{\sigma+1/2} = \sqrt{2}+\sqrt{2}\log(2)\sigma+\mathcal{O}(\sigma^2).\end{equation}
Plugging the equalities \eqref{first1}, \eqref{eq244} and \eqref{reszeta} into the function \eqref{eq243} we obtain the result of the statement.
\end{proof}
\begin{lem}\label{243}
The integral of \eqref{eq2463} vanishes, i.e.
$$\int_{\Gamma^{\infty}\setminus \mathcal{H}^{\hat{T}}}\left(\mathrm{CT}_{\sigma = 0}\lim_{T\to\infty}\int_{1}^Tv^{-\sigma-1}dv\right)y^sd\mu(z) = 0.$$
\end{lem}
\begin{proof}
It is straightforward that
$$\mathrm{CT}_{\sigma = 0}\lim_{T\to\infty}\int_{1}^Tv^{-\sigma-1}dv = 0.$$
\end{proof}
\begin{lem}\label{derivadaeis}
We obtain the following equality
\begin{align*}\mathrm{CT}_{s = 0}\int_{\Gamma^{\infty}\setminus \mathcal{H}^{\hat{T}}}y^s\log(y)\frac{dxdy}{y^2} =& \frac{\zeta'(0)}{\zeta(0)}\left(\zeta'(-1)-8\frac{\zeta'(0)}{\zeta(0)}+\zeta'(0)\right)+(2\zeta'(-1)+1)\gamma\\&-2\zeta'(-1)\log(\hat{T})-\frac{\log(\hat{T})+1}{\hat{T}}\Bigg],\end{align*}
where we recall that $\gamma$ is the Euler-Mascheroni constant.
\end{lem}
\begin{proof}
Let us suppose that $s\in\C$ satisfies $Re(s)\gg 0$, then
\[\frac{\partial}{\partial s}\int_{\Gamma^{\infty}\setminus \mathcal{H}^{\hat{T}}}y^s\frac{dxdy}{y^2} = \int_{\Gamma^{\infty}\setminus \mathcal{H}^{\hat{T}}}y^s\log(y)\frac{dxdy}{y^2}.\]
Then, using Proposition \ref{zagier} we obtain the following equalities
\begin{align}\int_{\Gamma^{\infty}\setminus \mathcal{H}^{\hat{T}}}y^s\log(y)\frac{dxdy}{y^2} &= \frac{\partial}{\partial s}\int_{X^{mod,\hat{T}}}E(\tau,s)\frac{dxdy}{y^2} \nonumber \\\label{derivative} &= \frac{\partial}{\partial s}\left(\hat{T}^{s-1}/(s-1)-\frac{\zeta^*(2s-1)}{\zeta^*(2s)}\hat{T}^{-s}/s\right).\end{align}
By direct computation we find that \eqref{derivative} satisfies the following equality 
\begin{align}\label{7}\frac{\partial}{\partial s}\left(\hat{T}^{s-1}/(s-1)-\frac{\zeta^*(2s-1)}{\zeta^*(2s)}\hat{T}^{-s}/s\right) =& \Big(\frac{\hat{T}^{s-1}((s-1)\log(\hat{T})-1)}{(s-1)^2}\\ &-\frac{\left(\zeta^{*\;'}(2s-1)\zeta^*(2s)-\zeta^*(2s-1)\zeta^{*\;'}(2s)\right)}{\zeta^{*}(2s)^2}\frac{\hat{T}^{-s}}{s}\nonumber\\ &+\frac{\zeta^*(2s-1)}{\zeta^*(2s)}\frac{T^{-s}(s\log(\hat{T})+1)}{s^2}\Big)\nonumber.\end{align}
The function on the right hand side is meromorphic. Using meromorphic continuation we can remove the hypothesis on $s$. We proceed by analyzing each factor of \eqref{7} separately. First we obtain  
\begin{equation}\label{aux1}\mathrm{CT}_{s = 0}\left(\frac{\hat{T}^{s-1}((s-1)\log(\hat{T})-1)}{(s-1)^2}\right) = -\frac{\log(\hat{T})+1}{\hat{T}}.\end{equation}
Furthermore, by direct computation 
\begin{align}\label{aux2}\mathrm{CT}_{s = 0}\left(\frac{\left(\zeta^{*\;'}(2s-1)\zeta^*(2s)-\zeta^*(2s-1)\zeta^{*\;'}(2s)\right)}{\zeta^{*}(2s)^2}\frac{\hat{T}^{-s}}{s}\right) =&  \frac{\zeta^{*'}(-1)}{\zeta(0)}+\zeta'(-1)\Big(-\frac{\log(\pi)}{2\zeta(0)}+2\gamma\\&+\frac{1}{\zeta(0)}\left(-\frac{2\zeta'(0)}{\zeta(0)}+\frac{\log(\pi)}{2}\right)\frac{\zeta'(0)}{\zeta(0)^2}\\&-2\log(\hat{T})\Big)-\frac{\zeta^{*'}(-1)}{\zeta(0)}.\end{align}
We can simplify the right hand side, obtaining 
\[\frac{\zeta'(0)}{\zeta(0)}\left(\zeta'(-1)-8\frac{\zeta'(0)}{\zeta(0)}+\zeta'(0)\right)+2\zeta'(-1)\gamma-2\zeta'(-1)\log(\hat{T}).\]
By direct computation 
\begin{equation}\label{aux3}\mathrm{CT}_{s = 0}\left(\frac{\zeta^*(2s-1)}{\zeta^*(2s)}\frac{T^{-s}(s\log(\hat{T})+1)}{s^2}\right) = \gamma,\end{equation}
where $\gamma$ is the Euler-Mascheroni constant. We conclude by plugging equations \eqref{aux1}, \eqref{aux2} and \eqref{aux3} into \eqref{7}.
\end{proof}
\begin{lem}\label{212}
The following equality holds:
$$\mathrm{CT}_{s = 0}\int_{\Gamma^{\infty}\setminus \mathcal{H}^{\hat{T}}}y^s\frac{dxdy}{y^2} = \frac{\pi}{3}-\hat{T}^{-1}.$$
\end{lem}
\begin{proof}
Let us suppose that $s\in\C$ satisfies that $\mathrm{Re}(s)\gg 0$. Using Definition \ref{defeis2} we get
\[\int_{\Gamma^{\infty}\setminus \mathcal{H}^{\hat{T}}}y^s\frac{dxdy}{y^2} = \int_{X^{mod,\hat{T}}}E(\tau,s)\frac{dxdy}{y^2}.\]
Proceeding as in the proof of the previous Lemma we use Proposition \ref{zagier} to obtain
\begin{equation}\label{eq249}\int_{X^{mod,\hat{T}}}E(\tau,s)\frac{dxdy}{y^2} = \left(\hat{T}^{s-1}/(s-1)-\frac{\zeta^*(2s-1)}{\zeta^*(2s)}\hat{T}^{-s}/s\right).\end{equation}
The right hand side is a meromorphic function on the variable $s$. Then by meromorphic continuation we remove the hypothesis on $s$. On the one hand
\begin{equation}\label{eq247}\mathrm{CT}_{s = 0}\frac{\zeta^*(2s-1)}{\zeta^*(2s)}\hat{T}^{-s}/(-s) = \frac{\pi}{3}.\end{equation}
On the other hand
\begin{equation}\label{eq248}\mathrm{CT}_{s = 0}\hat{T}^{s-1}/(s-1) = -\hat{T}^{-1}.\end{equation}
We conclude by plugging the equalities \eqref{eq247}, \eqref{eq248} into \eqref{eq249}.
\end{proof}
\begin{lem}\label{221}
We obtain 
\begin{align*}\mathrm{CT}_{s = 0}\int_{\Gamma^{\infty}\setminus \mathcal{H}^{\hat{T}}}y^s\frac{dydx}{y} = \frac{3}{\pi}\left(\gamma+\log\left(\frac{\pi}{4}\right)+\frac{\zeta^{*'}(2)}{\zeta^*(2)}\right)\frac{1}{\hat{T}}-\frac{1}{2\zeta^*(2)}\frac{\log(\hat{T})+1}{\hat{T}}+\log(\hat{T}).\end{align*}
\end{lem}
\begin{proof}
The following equality follows directly
\begin{equation}\label{eq252}\mathrm{CT}_{s = 0}\int_{\Gamma^{\infty}\setminus \mathcal{H}^{\hat{T}}} y^s\frac{dxdy}{y} = \mathrm{CT}_{s = 1}\int_{\Gamma^{\infty}\setminus \mathcal{H}^{\hat{T}}} y^s\frac{dxdy}{y^2}.\end{equation}
Let $s\in\C$ satisfying that $\mathrm{Re}(s)\gg 0$. By applying the same reasoning used in the proofs of the previous two lemmas, we use Proposition \ref{zagier} to rewrite the above equality as follows:
\begin{equation}\label{eq251}\int_{X^{mod,\hat{T}}} E(\tau,s)\frac{dxdy}{y^2} = \left(\hat{T}^{s-1}/(s-1)-\frac{\zeta^*(2s-1)}{\zeta^*(2s)}\hat{T}^{-s}/s\right).\end{equation}
The right hand side of the equality is meromorphic. Then we apply meromorphic continuation to remove the hypothesis on $s$. On the one hand we obtain
\begin{equation}\label{eq253}\mathrm{CT}_{s = 1}\frac{\zeta^*(2s-1)}{\zeta^*(2s)}\hat{T}^{-s}/{-s} = \mathrm{CT}_{s = 1}\left(\frac{\zeta^*(2s-1)}{\zeta^*(2s)}\right)\frac{1}{\hat{T}}-\mathrm{Res}_{s = 1}\left(\frac{\zeta^*(2s-1)}{\zeta^*(2s)}\right)\frac{\log(\hat{T})+1}{\hat{T}}.\end{equation}
By direct computation
\begin{align*}\mathrm{CT}_{s = 1}\left(\frac{\zeta^*(2s-1)}{\zeta^*(2s)}\right) &= \frac{\pi^{-1/2}}{\zeta^*(2)}\left(\gamma\Gamma(1/2)+\frac{1}{2}\left(\log(\pi)\Gamma(1/2)+\Gamma'(1/2)\right)+\frac{\zeta^{*'}(2)\Gamma(1/2)}{2\zeta^*(2)}\right),\\
\mathrm{Res}_{s = 1}\left(\frac{\zeta^*(2s-1)}{\zeta^*(2s)}\right) & = \frac{\pi^{-1/2}\Gamma(1/2)}{2\zeta^*(2)}.
\end{align*}
On the other hand
\begin{equation}\label{eq254}\mathrm{CT}_{s = 1} \hat{T}^{s-1}/(s-1) = \log(\hat{T}).\end{equation}
Using the equalities \eqref{eq251}, \eqref{eq253} and \eqref{eq254} in \eqref{eq252}, the result follows.
\end{proof}
\begin{cor}\label{213}
It holds that
\begin{align*}\int_{X^{mod,\hat{T}}}&\left(\mathrm{CT}_{\sigma = 0}\lim_{T\to\infty}\int_{\mathcal{F}^T_2}f_{\mu_0}(\tau)\vartheta(v,z,\mu_0)_0v^{-\sigma}d\mu(\tau)\right)d\mu(z) = c_{\mu_0}(0)\Bigg(A\left(\frac{\pi}{3}-\hat{T}^{-1}\right) \\ & -2\erf\left(\sqrt{\frac{\pi}{2}}\right)\Bigg[\frac{\zeta'(0)}{\zeta(0)}\left(\zeta'(-1)-8\frac{\zeta'(0)}{\zeta(0)}+\zeta'(0)\right)+(2\zeta'(-1)+1)\gamma\\&-2\zeta'(-1)\log(\hat{T})-\frac{\log(\hat{T})+1}{\hat{T}}\Bigg]\\ &+\frac{3}{\pi}\left(\gamma+\log\left(\frac{\pi}{4}\right)+\frac{\zeta^{*'}(2)}{\zeta^*(2)}\right)\frac{1}{\hat{T}}-\frac{1}{2\zeta^*(2)}\frac{\log(\hat{T})+1}{\hat{T}}+\log(\hat{T})\Bigg).\end{align*}
\end{cor}
\begin{proof}
The proof follows by applying Lemmas \ref{derivadaeis}, \ref{212} and \ref{221} to Proposition \ref{227}.
\end{proof}
\subsection{Main result}
The goal of this paper is to compute 
\begin{align}\label{eq227}\int_{X^{mod,\hat{T}}}&\log||\Psi(f)(z)||_{Pet}d\mu(z) \\ &= -\frac{1}{4}\int_{X^{mod,\hat{T}}}\Phi(f)(z) d\mu(z)-\frac{c_{\mu_0}(0)\vol(X^{mod,\hat{T}})}{2}\left(-\gamma/2+\log\sqrt{2\pi}\right).\nonumber\end{align}
In \S\ref{FactorizationIntegralsSection}, we split the above computation into the sum of two types of integrals, which we evaluated using distinct methods in \S\ref{216} and \S\ref{219}. From these results, we obtain the following asymptotic expansion.
\begin{thm}[Main result]\label{mainresult}
Let $f(\tau) =  \sum_{\substack{n\in\Z\\j \in\{0,1\}}}c_{\mu_j}(n)q^n\varphi_{\mu_j}$ be a vector valued weakly holomorphic modular form of weight $1/2$. The following equality holds:
\begin{align*}
\int_{X^{mod,\hat{T}}}&\log||\Psi(f)(z)||_{Pet}d\mu(z) = -\frac{\vol(X^{mod})}{2}\suma\sum_{m\in j/2 +\Z^{\geq 0}} c_{\mu_j}(-m)\kappa_{\mu_j}(m)\\&-\frac{c_{\mu_0}(0)}{4}\log(\hat{T})\left(1-2\mathrm{vol}(X^{mod})+4\mathrm{erf}\left(\sqrt{\frac{\pi}{2}}\right)\zeta'(-1)\right)\\&+O(e^{-\hat{T}})-\frac{c_{\mu_0}}{4}(0)\left(\frac{\log(\hat{T})+1}{\hat{T}}\right)\left(2\mathrm{erf}\left(\sqrt{\frac{\pi}{2}}\right)-\frac{1}{2\zeta^*(2)}\right)\\ &-\frac{c_{\mu_0}(0)\mathrm{vol}(X^{mod})^{-1}}{4\hat{T}}\Bigg(-\frac{\pi A}{3}+\left(\gamma+\log\left(\frac{\pi}{4}\right)+\frac{\zeta^{*'}(2)}{\zeta^*(2)}\right)+\left(-\gamma/2+\log\sqrt{2\pi}\right)\Bigg),
\end{align*}
with
\[\kappa_{\mu_j}(m) := \left\{\begin{matrix}b(m,\varphi_{\mu_j})&if\;m\neq0\\0& if\;m = 0\;and,\;j = 1\\ C&if\;m = 0\;and,\;j = 0\end{matrix}\right.\]
where 
\begin{align*}C :=& A/2-\erf\left(\sqrt{\frac{\pi}{2}}\right)\Bigg[\frac{\zeta'(0)}{\zeta(0)}\left(\zeta'(-1)-8\frac{\zeta'(0)}{\zeta(0)}+\zeta'(0)\right)+(2\zeta'(-1)+1)\gamma\bigg] \\&-\tanh^{-1}\left(\frac{\sqrt{7}}{4}\right)+\frac{\sqrt{7}}{4}+ \frac{\log(3/4)}{2}-\gamma/2+\log\sqrt{2\pi}.\end{align*}
\end{thm}
\begin{proof}
Using the equality \eqref{eq227}, Lemma \ref{lema513} and Corollaries \ref{213} and \ref{218}, the result follows.
\end{proof}
\begin{cor}\label{MainCorollary}
    Let $f(\tau) =  \sum_{\substack{n\in\Z\\j \in\{0,1\}}}c_{\mu_j}(n)q^n\varphi_{\mu_j}$ be a vector valued weakly holomorphic modular form of weight $1/2$ such that $c_{\mu_0}(0) = 0$. Then 
    \[\int_{X^{mod}}\log||\Psi(f)(z)||_{Pet}d\mu(z) = -\frac{\vol(X^{mod})}{2}\suma\sum_{m\in j/2 +\Z^{> 0}} c_{\mu_j}(-m)\kappa_{\mu_j}(m).\]
\end{cor}
\begin{proof}
   This follows directly by taking the limit as $\hat{T}\to \infty$ in Theorem \ref{mainresult}.
\end{proof}
\section{Auxiliary computations}
\subsection{Divergence}
Let $V$ be the quadratic space and $L$ the lattice defined in \S\ref{modexample}. We set $f(\tau) = \suma f_{\mu_j}(\tau)\varphi_{\mu_j}\in M_{1/2,L}^!$. The main goal of this subsection is to understand the integrals
$$\suma\int_{\mathcal{F}^T}f_{\mu_j}(\tau)\left(\int_{X^{mod,\hat{T}}}\mathbf{Div}(\tau,z,\mu_j)d\mu(z)\right)d\mu(\tau).$$

First, we obtain an explicit expression for $\mathbf{Div}(\tau,z,\mu_j)$, which requires a clear understanding of the classical embedding \begin{equation}\label{embedpoinc}\mathcal{H}\hookrightarrow SO(V)(\R).\end{equation}
The map \eqref{embedpoinc} is defined as the composition of 
\begin{align*}
    \mathcal{H}\to& \mathrm{SL}_2(\R)\\
    z = x+iy\mapsto&\begin{pmatrix}y^{1/2}& xy^{-1/2}\\ & y^{-1/2}\end{pmatrix},
\end{align*}
together with the $2$ to $1$ morphism $\mathrm{SL}_2(\R)\to  SO(V)(\R)$, which has kernel $-I_2$. Specifically, the group $SO(V)(\R)$ acts transitively on $V(\R)$ via the standard action, while $\mathrm{SL}_2(\R)$ acts transitively on $V(\R)$ by conjugation. These two actions together define $\mathrm{SL}_2(\R)\to  SO(V)(\R)$. Throughout this subsection we will use the map \eqref{embedpoinc} without referring to it.
\begin{lem}\label{234}
The divergent part satisfies 
$$\mathbf{Div}(\tau,z,\mu_j) = y\mathbf{Div}(\tau,z_0,\mu_j),$$
where we recall that $z_0$ was the element defined in \ref{GaussianDef}.
\end{lem}
\begin{proof}
Using the invariance property of the Gaussian, the following equality holds:
$$\mathrm{Div}(g_{\tau},h_1,\varphi^{\infty}_{z,\mu_j}) = \sum_{x_0\in \Z+\frac{j}{2}}\int_{\R}\omega(g_{\tau},id)\varphi^{\infty}_{z,\mu_j}\begin{pmatrix}x_{\R}\\x_0\\0\end{pmatrix}dx_{\R} = \sum_{x_0\in \Z+\frac{j}{2}}\int_{\R}\omega(g_{\tau},h_{z})\varphi^{\infty}_{z_0,\mu_j}\begin{pmatrix}x_{\R}\\x_0\\0\end{pmatrix}dx_{\R},$$
where $h_1$ and $h_{z}$ are the image of $1$ and $z$ under the map \eqref{embedpoinc}. Applying the Weil representation 
\begin{align*}\sum_{x_0\in \Z+\frac{j}{2}}\int_{\R}\omega(g_{\tau},h_{z})\varphi^{\infty}_{z_0,\mu_j}\begin{pmatrix}x_{\R}\\x_0\\0\end{pmatrix}dx_{\R} &= \sum_{x_0\in \Z+\frac{j}{2}}\int_{\R}\omega(g_{\tau},h_{iy})\varphi^{\infty}_{z_0,\mu_j}\begin{pmatrix}x_{\R}-2xx_0\\x_0\\0\end{pmatrix}dx_{\R} \\&= \sum_{x_0\in \Z+\frac{j}{2}}\int_{\R}\omega(g_{\tau},h_{iy})\varphi^{\infty}_{z_0,\mu_j}\begin{pmatrix}x_{\R}\\x_0\\0\end{pmatrix}dx_{\R}.
\end{align*}
Where the last equality follows by the change of variables $x_{\R}-2xx_0 = x_{\R}$. Moreover, we have 
\begin{align*}
\sum_{x_0\in \Z+\frac{j}{2}}\int_{\R}\omega(g_{\tau},h_{iy})\varphi^{\infty}_{z_0,\mu_j}\begin{pmatrix}x_{\R}\\x_0\\0\end{pmatrix}dx_{\R} = \sum_{x_0\in \Z+\frac{j}{2}}\int_{\R}\omega(g_{\tau})\varphi^{\infty}_{z_0,\mu_j}\begin{pmatrix}y^{-1}x_{\R}\\x_0\\0\end{pmatrix}dx_{\R}\end{align*}
By a change of variable of the form 
$y^{-1}x_{\R} = x_{\R}$, we obtain
$$\sum_{x_0\in \Z+\frac{j}{2}}\int_{\R}\omega(g_{\tau})\varphi^{\infty}_{z_0,\mu_j}\begin{pmatrix}y^{-1}x_{\R}\\x_0\\0\end{pmatrix}dx_{\R} = y\sum_{x_0\in \Z+\frac{j}{2}}\int_{\R}\omega(g_{\tau})\varphi^{\infty}_{z_0,\mu_j}\begin{pmatrix}x_{\R}\\x_0\\0\end{pmatrix}dx_{\R},$$
which implies the result.
\end{proof}
We now divide the computation of the main integral of this subsection into two parts, depending on the behavior of $\mathbf{Div}(\tau,z,\mu_j)$ on the variable $z$. Set 
\begin{align*}X^{mod}_0 &:= \{z = x+iy\in \mathcal{H},\;s.t.\;|x|\leq 1/2,\;|z|\geq 1,\;y\leq 1\},\\
X^{mod,\hat{T}}_1&:= \{z = x+iy\in \mathcal{H},\;s.t.\;|x|\leq 1/2,\;|z|\geq 1, 1<y\leq \hat{T}\},\end{align*}
so that 
\begin{equation}\label{eq188}X^{mod,\hat{T}}= X^{mod}_0\bigsqcup X^{mod,\hat{T}}_1.\end{equation}
\begin{lem}\label{115}
The following equality holds: 
$$\int_{X^{mod,\hat{T}}_1}\mathbf{Div}(\tau,z,\mu_j)d\mu(z) = \log(\hat{T})\mathbf{Div}(\tau,z_0,\mu_j).$$
\end{lem}
\begin{proof}
Using Lemma \ref{234} 
\begin{align*}\int_{X^{mod,\hat{T}}_1}\mathbf{Div}(\tau,z,\mu_j)d\mu(z) &= \int_1^{\hat{T}}\int_{-1/2}^{1/2}y\mathbf{Div}(\tau,z_0,\mu_j)\frac{dxdy}{y^2} \\ &=\mathbf{Div}(\tau,z_0,\mu_j)\int_1^{\hat{T}}\frac{1}{y}dy = \log(\hat{T})\mathbf{Div}(\tau,z_0,\mu_j).\end{align*}
\end{proof}
\begin{lem}\label{184}
We obtain
$$\int_{ X^{mod}_0}\mathbf{Div}(\tau,z,\mu_j)d\mu(z) = 2\mathbf{Div}(\tau,z_0,\mu_j)\left(-\tanh^{-1}\left(\frac{\sqrt{7}}{4}\right)+\frac{\sqrt{7}}{4}+ \frac{\log(3/4)}{2}\right).$$
\end{lem}
\begin{proof}
To simplify the computation we factor $ X^{mod}_0$ into the following two subsets:
$$A = \{z = x+iy\in \mathcal{H},\;s.t.\;|x|\leq 1/2,\;|z|\geq 1,\;y\leq 1\;x\leq 0\},$$
and 
$$B = \{z = x+iy\in \mathcal{H},\;s.t.\;|x|\leq 1/2,\;|z|\geq 1\;y\leq 1,\;x > 0\}.$$
The truncated fundamental domain of the modular curve satisfies
$$ X^{mod}_0 = A\bigsqcup B.$$
The proof of Lemma \ref{234} shows that $\mathbf{Div}(\tau,z,\mu_j)$ does not depend on the variable $x$. Then
\begin{align}\int_{ X^{mod}_0}\mathbf{Div}(\tau,z,\mu_j)d\mu(z) &= \int_{A}\mathbf{Div}(\tau,z,\mu_j)d\mu(z)+\int_{B}\mathbf{Div}(\tau,z,\mu_j)d\mu(z)  \nonumber \\ \label{eq235}&=2\int_{A}\mathbf{Div}(\tau,z,\mu_j)d\mu(z).\end{align}
We apply Lemma \ref{234} to the integral \eqref{eq235}, then 
\begin{align*}\int_{A}\mathbf{Div}(\tau,z,\mu_j)d\mu(z) &= \int_{3/4}^1\int_{-1/2}^{-\sqrt{1-y^2}}\frac{1}{y}\mathbf{Div}(\tau,z_0,\mu_j)dxdy  \\
 &= \mathbf{Div}(\tau,z_0,\mu_j)\left(\int_{3/4}^1\frac{-\sqrt{1-y^2}}{y}dy+\frac{1}{2}\int_{3/4}^1\frac{1}{y}dy\right) = \\
&= \mathbf{Div}(\tau,z_0,\mu_j)\left(-\tanh^{-1}\left(\frac{\sqrt{7}}{4}\right)+\frac{\sqrt{7}}{4}+ \frac{\log(3/4)}{2}\right).\end{align*}
\end{proof}
\begin{prop}\label{185}
The following equality holds:
$$\mathbf{Div}(\tau,z_0,\mu_0) = v^{1/2}\theta^{Jac}_{\mu_0}(\tau),$$
where $\theta^{Jac}_{\mu_0}(\tau) = \sum_{n\in \Z}e^{2\pi in^2\tau}$ is the Jacobi theta function. Furthermore $$\mathbf{Div}(\tau,z_0,\mu_1)= v^{1/2}\sum_{n\in\frac{1}{2}+\Z}e^{2\pi i n^2\tau} =:v^{1/2}\theta^{Jac}_{\mu_1}(\tau)$$
\end{prop}
\begin{proof}
We prove the first statement, the second one follows similarly. Applying the Weil representation one obtain 
\begin{align}\mathbf{Div}(\tau,z_0,\mu_0) &= v^{1/4}\sum_{x_0\in \Z}\int_{\R}\omega(g_{\tau})\Gaussr\begin{pmatrix}x_{\R}\\x_0\\0\end{pmatrix}dx_{\R} \nonumber\\
\label{eq255} &= v\sum_{x_0\in \Z}\int_{\R}\psi_{\infty}\left(2uq\left(x_{\R},x_0,0\right)\right)\Gaussr\begin{pmatrix}v^{1/2}x_{\R}\\v^{1/2}x_0\\0\end{pmatrix}dx_{\R}.\end{align}
We recall that in the present case $z_0 = i$. Then
$$\psi_{\infty}\left(2uq\left(x_{\R},x_0,0\right)\right)\Gaussr\begin{pmatrix}v^{-1/2}x_{\R}\\v^{1/2}x_0\\0\end{pmatrix} = e^{-\pi\left(vx_{\R}^2+2vx_0^2\right)+2\pi i u x_0^2}.$$
By direct computation 
\begin{equation}\label{eq256} v\int_{\R}e^{-\pi\left(vx_{\R}^2+2vx_0^2\right)+2\pi i u x_0^2}dx_{\R} = v^{1/2}e^{2\pi i x_0^2(u+iv)}.\end{equation}
Applying equality \eqref{eq256} in \eqref{eq255} we obtain
$$\mathbf{Div}(\tau,z_0,\mu_0) = v^{1/2}\sum_{x_0\in\Z}e^{2\pi i x_0^2\tau} =v^{1/2}\theta^{Jac}_{\mu_0}(\tau).$$

\end{proof}
\begin{lem}\label{eistheta}
The function $v^{1/2}\theta^{Jac}_{\mu_1}(\tau) = v^{1/2}\sum_{n\in\mu_1+\Z}e^{2\pi i n^2\tau}$ is a non-holomorphic modular form of weight $1/2$.
\end{lem}
\begin{proof}
According to \cite[Proposition\;6.3,\;p.\;2301]{kudla_2010_eisenstein}
\[\sum_{n\in\mu_1+\Z}e^{2\pi i n^2\tau} = E(\tau,-1/2,1/2,\mu(\varphi_{\mu_1}))).\]
Using \cite[Lemma\;1.1,\;p.\;11]{kudla_2003_integral} the result holds.
\end{proof}
\begin{prop}\label{226}
We obtain
\begin{align*}\suma&\int_{\mathcal{F}^T}f_{\mu_j}(\tau)\left(\int_{X^{mod,\hat{T}}}\mathbf{Div}(\tau,z,\mu_j)d\mu(z)\right)d\mu(\tau)\\& =-\frac{1}{\sqrt{T}}\suma2c_{\mu_j}(0)\left(\log(\hat{T})+2\left(-\tanh^{-1}\left(\frac{\sqrt{7}}{4}\right)+\frac{\sqrt{7}}{4}+ \frac{\log(3/4)}{2}\right)\right).\end{align*}
\end{prop}
\begin{proof}
We factor the integral of the statement according to \eqref{eq188}:
\begin{align}\label{eq226}\int_{\mathcal{F}^T}f_{\mu_j}(\tau)\left(\int_{X^{mod,\hat{T}}}\mathbf{Div}(\tau,z,\mu_j)d\mu(z)\right)d\mu(\tau) =& \int_{\mathcal{F}^T}f_{\mu_j}(\tau)\left(\int_{ X^{mod}_0}\mathbf{Div}(\tau,z,\mu_j)d\mu(z)\right)d\mu(\tau) \\
&+\int_{\mathcal{F}^T}f_{\mu_j}(\tau)\left(\int_{X^{mod,\hat{T}}_1}\mathbf{Div}(\tau,z,\mu_j)d\mu(z)\right)d\mu(\tau).\nonumber\end{align}
By Lemmas \ref{115} and \ref{184}, the function \eqref{eq226} is equal to 
\begin{equation}\label{eq138}\left(\log(\hat{T})+2\left(-\tanh^{-1}\left(\frac{\sqrt{7}}{4}\right)+\frac{\sqrt{7}}{4}+ \frac{\log(3/4)}{2}\right)\right)\int_{\mathcal{F}^T}f_{\mu_j}(\tau)\mathbf{Div}(\tau,z_0,\mu_j)d\mu(\tau).\end{equation}
Propositions \ref{215} and \ref{185} imply
\begin{align}\label{eq137}\suma\int_{\mathcal{F}^T}f_{\mu_j}(\tau)\mathbf{Div}(\tau,z,\mu_j)d\mu(\tau) = -\suma\frac{2c_{\mu_j}(0)}{\sqrt{T}}.
\end{align}
Plugging the equality \eqref{eq137} into the function \eqref{eq138} we obtain the result.
\end{proof}
\begin{prop}\label{232}
The divergent part satisfies
\begin{align*}\int_{\mathcal{F}^T_2}&f_{\mu_0}(\tau)\left(\int_{X^{mod,\hat{T}}}\mathbf{Div}(v,z,\mu_0)_0d\mu(z)\right)d\mu(\tau) \\&=-2c_{\mu_0}(0)\left(1-\sqrt{T}^{-1}\right)\left(\log(\hat{T}) +2\left(-\tanh^{-1}\left(\frac{\sqrt{7}}{4}\right)+\frac{\sqrt{7}}{4}+ \frac{\log(3/4)}{2}\right)\right).\end{align*}
\end{prop}
\begin{proof}
We proceed as in the proof of Proposition \ref{226}. First of all we observe
\begin{equation}\label{eqlem67}\int_{X^{mod,\hat{T}}}\mathbf{Div}(v,z,\mu_0)_0d\mu(z) = \int_{X^{mod,\hat{T}}}\left(\intfour\mathbf{Div}(u+iv,z,\mu_0)du\right)d\mu(z).\end{equation}
By means of Fubini's Theorem it holds that \eqref{eqlem67} is equal to 
\begin{equation}\label{eq212}\intfour\left(\int_{X^{mod,\hat{T}}}\mathbf{Div}(u+iv,z,\mu_0)d\mu(z)\right)du,\end{equation}
which is the Fourier constant term of the function
$$\tau\mapsto\int_{X^{mod,\hat{T}}}\mathbf{Div}(\tau,z,\mu_0)d\mu(z).$$
We factor the integral over $X^{mod,\hat{T}}$ according to \eqref{eq118}. Applying Lemmas \ref{115} and \ref{184} to the function \eqref{eq212} we obtain 
\begin{align}\intfour&\left(\int_{X^{mod,\hat{T}}}\mathbf{Div}(u+iv,z,\mu_0)d\mu(z)\right)du \nonumber\\\label{eq233} &= \left(\log(\hat{T})+2\left(-\tanh^{-1}\left(\frac{\sqrt{7}}{4}\right)+\frac{\sqrt{7}}{4}+ \frac{\log(3/4)}{2}\right)\right)\mathbf{Div}(v,z_0,\mu_0)_0.\end{align}
Using Proposition \ref{185} in \eqref{eq233}
\begin{equation*}\int_{\mathcal{F}^T_2}f_{\mu_0}(\tau)\mathbf{Div}(v,z,\mu_0)_0d\mu(\tau) = \int_{\mathcal{F}^T_2}v^{1/2}f_{\mu_0}(\tau)\theta^{Jac}_{\mu_0}(v)_{0}d\mu(\tau),\end{equation*}
where $\theta^{Jac}_{\mu_0}(v)_{0} := \intfour\theta^{Jac}_{\mu_0}(u+iv)du = 1$. Moreover, by Corollary \ref{integralsola}
\begin{align}\label{eq232}
\int_{\mathcal{F}^T_2}v^{1/2}f_{\mu_0}(\tau)\theta^{Jac}_{\mu_0}(v)_{0}d\mu(\tau) = 2c_{\mu_0}(0)- \frac{2c_{\mu_0}(0)}{\sqrt{T}}
\end{align}
Lastly, we plugg the equality \eqref{eq232} into the function \eqref{eq233} to obtain the statement.
\end{proof}
\subsection{Two integrals}\label{231}
This subsection is devoted to compute the integrals
\begin{equation}\label{eq102}\suma\int_{\mathcal{F}^T}v^{1/2}f_{\mu_j}(\tau)\theta^{Jac}_{\mu_j}(\tau)d\mu(\tau),\end{equation}
\begin{equation*}\suma\int_{\mathcal{F}^T}f_{\mu_j}(\tau)A_{-1}(\tau,-1/2,\mu(\varphi_{\mu_j}))d\mu(\tau),\end{equation*}
along with several variants of these calculations. The following discussion is based on the techniques developed in \cite[\S2,\;p.\;16]{kudla_2003_integral}.
\begin{prop}\label{215}
The integral \eqref{eq102} satisfies the following equality: 
$$\suma\int_{\mathcal{F}^T}v^{1/2}f_{\mu_j}(\tau)\theta^{Jac}_{\mu_j}(\tau)d\mu(\tau) = \suma\frac{-2c_{\mu_j}(0)}{\sqrt{T}}$$
\end{prop}
\begin{proof}
We have
\begin{equation}\label{eq257}\int_{\mathcal{F}^T}v^{1/2}f_{\mu_j}(\tau)\theta^{Jac}_{\mu_j}(\tau)d\mu(\tau) = \int_{\mathcal{F}^T}v^{-3/2}f_{\mu_j}(\tau)\theta^{Jac}_{\mu_j}(\tau)dudv.\end{equation}
The Jacobi theta function $\theta^{Jac}_{\mu_j}(\tau)$ and $f_{\mu_j}(\tau)$ are holomorphic functions at $\tau\in\mathcal{H}$, then 
$$\frac{\partial}{\partial\overline{\tau}}\theta^{Jac}_{\mu_j}(\tau)f_{\mu_j}(\tau) = 0.$$
The previous equality allow us to obtain a preimage of 
$$v^{-3/2}f_{\mu_j}(\tau)\theta^{Jac}_{\mu_j}(\tau),$$
under the operator $\frac{\partial}{\partial\overline{\tau}}$, in fact by direct computation
\begin{equation}\label{eq236}\frac{2}{i}\frac{\partial}{\partial\overline{\tau}}\left\{\frac{-2}{\sqrt{v}}f_{\mu_j}(\tau)\theta^{Jac}_{\mu_j}(\tau)\right\} =  v^{-3/2}f_{\mu_j}(\tau)\theta^{Jac}_{\mu_j}(\tau).\end{equation}
We apply Stokes Theorem to \eqref{eq257}. By \eqref{eq236} we obtain
$$\int_{\mathcal{F}^T}v^{-3/2}f_{\mu_j}(\tau)\theta^{Jac}_{\mu_j}(\tau)dudv = \frac{2}{i}\frac{1}{2i}\int_{\partial\mathcal{F}^T}\frac{-2}{\sqrt{v}}f_{\mu_j}(\tau)\theta^{Jac}_{\mu_j}(\tau)d\tau = \int_{\partial\mathcal{F}^T}\frac{2}{\sqrt{v}}f_{\mu_j}(\tau)\theta^{Jac}_{\mu_j}(\tau)d\tau.$$
Given $\tau\in\mathcal{H}$ such that $|\tau| = 1$, the function $\mathrm{Im}(\tau)$ is invariant under the transformation $$\tau\mapsto -1/\tau.$$
Moreover, the function $\mathrm{Im}(\tau)$ is invariant under the transformation 
$$\tau\to\tau+1.$$
The same properties are satisfied by $\frac{1}{v^{1/2}}$. According to Lemma \ref{eistheta} and \cite[(1.42),\;p.\;13]{kudla_2003_integral} the function \[v^{-1/2}\suma f_{\mu_j}(\tau)\theta^{Jac}_{\mu_j}(\tau),\] is invariant under $\mathrm{SL}_2(\Z)$ and then it is invariant under the aforementioned transformations. Then
\[\frac{2}{\sqrt{v}}f_{\mu_j}(\tau)\theta^{Jac}_{\mu_j}(\tau)d\tau,\]
is invariant under $\tau\mapsto -1/\tau$ and $\tau\mapsto \tau+1$. The above discussion implies the following equality:
$$\suma\int_{\partial\mathcal{F}^T}\frac{2}{\sqrt{v}}f_{\mu_j}(\tau)\theta^{Jac}_{\mu_j}(\tau)d\tau =\suma\left(\int_{1/2}^{-1/2}\frac{2}{\sqrt{v}}f_{\mu_j}(\tau)\theta^{Jac}_{\mu_j}(\tau)d\tau\right)_{v = T}.$$
Using the definition of the constant term of the Fourier expansion we can go further, concluding the proof:
$$\suma\left(\int_{1/2}^{-1/2}\frac{2}{\sqrt{v}}f_{\mu_j}(\tau)\theta^{Jac}_{\mu_j}(\tau)d\tau\right)_{v = T}  = \suma\frac{-2c_{\mu_j}(0)}{\sqrt{T}}.$$
\end{proof}
\begin{cor}\label{integralsola}
Let $f(\tau) = \suma f_{\mu_j}(\tau)\varphi_{\mu_j}$ be a weakly holomorphic modular form, then
\[\suma \int_{\mathcal{F}^T_2}v^{1/2}f_{\mu_j}(\tau)d\mu(\tau) =\suma 2c_{\mu_j}(0)- \frac{2c_{\mu_j}(0)}{\sqrt{T}}.\]
\end{cor}
\begin{proof}
The functions $f_{\mu_j}(\tau)$ are holomorphic and $\tau\mapsto \tau+1$ invariant. Using the Stokes argument of the proof of Proposition \ref{215} we obtain the statement.
\end{proof}
\begin{lem}\label{230}
We obtain
$$\suma\int_{\mathcal{F}^T}f_{\mu_j}(\tau)\RESE(\tau,-1/2,\mu(\varphi_{\mu_j}))d\mu(\tau) = 2\sum_{j = 0}^1\sum_{m\in j/2+\Z^{\geq 0}}c_{\mu_j}(-m)b(m,T,\mu(\varphi_{\mu_j})).$$
\end{lem} 
\begin{proof}
The proof follows as in \cite[p.\;21]{kudla_2003_integral}. Nevertheless for the sake of completeness we will give an overview of the proof. Proposition \ref{107} allows us to apply Stokes Theorem as in the proof of Proposition \ref{215}, obtaining 
$$\suma \int_{\mathcal{F}^T_2}f_{\mu_j}(\tau)E(\tau,s,-1/2,\mu(\varphi_{\mu_j}))d\mu(\tau) =  \frac{-2}{(s-1/2)}\int_{\partial\mathcal{F}^T}\suma f_{\mu_j}(\tau)E(\tau,s,3/2,\mu(\varphi_{\mu_j}))d\tau.$$
By \cite[(1.42),\;p.\;13]{kudla_2003_integral}, the function $\suma f_{\mu_j}(\tau)E(\tau,s,3/2,\mu(\varphi_{\mu_j}))$ is $\mathrm{SL}_2(\Z)-$invariant. In particular it is invariant under the transformations 
$$\tau\mapsto \tau+1,\;\tau\mapsto -1/\tau.$$
Hence
\begin{align*}\suma \frac{-2}{(s-1/2)}\int_{\partial\mathcal{F}^T}f_{\mu_i}(\tau)E(\tau,s,3/2,\mu(\varphi_{\mu_j}))d\tau &=\sum_{i = 0}^1 \frac{-2}{(s-1/2)}\int_{1/2+iT}^{-1/2+iT}f_{\mu_j}(\tau)E(\tau,s,3/2,\mu(\varphi_{\mu_j}))d\tau  \\ &= \suma \frac{2}{(s-1/2)}\left(f_{\mu_j}(\tau)E(\tau,s,3/2,\mu(\varphi_{\mu_j}))\right)_{0,v = T}.\end{align*}
Taking the residue and using \cite[Corollary\;2.5,\;p.\;2283]{kudla_2010_eisenstein}, the functions $b(m,s,T,\mu(\tilde{\varphi_j}))$ are holomorphic at $s = 1/2$, hence the residue vanishes for all $m\in\Z$.
\end{proof}
\begin{cor}\label{238}
The following equality holds:
\begin{align*}\int_{\mathcal{F}_2^T}& f_{\mu_0}(\tau)\CTE(v,-1/2,\mu(\varphi_{\mu_0}))_0d\mu(\tau) \\=-&c_{\mu_0}(0)\Bigg(\log(T)+\frac{(-1)^{1/4}\sqrt{2}\pi}{T^{1/4}}\left(2\log(T/16)+16-8\gamma\right)\Bigg).\end{align*}
\end{cor}
\begin{proof}
By Proposition \ref{ConstantTermsEisenstein}, we have
\begin{align*}\CTE(v,-1/2,\mu(\varphi_{\mu_0}))_0 &= v+v^{3/4}(-1)^{1/4}\sqrt{2}\pi\left(\frac{\log(v/16)+4}{2}-2\gamma\right).\end{align*}
Using Stokes Theorem as in Proposition \ref{215}, we conclude the result.
\end{proof}
\begin{lem}\label{lemholom}
Let $\tilde{\varphi}_0\in\mathcal{S}(V(\A))$ be a $\prod_{p\nmid\infty}\mathrm{SL}_2(\Z_p)-$invariant function, then
\[\suma\int_{\mathcal{F}^T}f_{\mu_j}(\tau)\RESE(\tau,-1/2,\mu(\tilde{\varphi}_0))v^{\sigma}d\mu(\tau) = 0.\]
\end{lem}
\begin{proof}
By \cite[(1.42),\;p.\;13]{kudla_2003_integral}, the function $\suma f_{\mu_j}(\tau)\RESE(\tau,-1/2,\mu(\tilde{\varphi}_0))$ is invariant under $\tau\mapsto\tau+1$ and $\tau\mapsto -1/\tau$. Hence using Proposition \ref{232} and Proposition \ref{230} we obtain
\[\suma\int_{\mathcal{F}^T}f_{\mu_j}(\tau)\RESE(\tau,-1/2,\mu(\tilde{\varphi}_0))v^{\sigma}d\mu(\tau) = \sum_{i = 0}^12\mathrm{Res}_{s = 1/2}\sum_{m\in j/2+\Z}c_{\mu_i}(-m)\tilde{b}(m,s,T,\mu(\tilde{\varphi}_0)),\]
where $\tilde{b}(m,s,T,\mu(\tilde{\varphi}_0))$ denotes the first term of the Laurent series expansion of the $m$-th Fourier coefficient of $\RESE(\tau,-1/2,\mu(\tilde{\varphi}_0))$. By \cite[Corollary\;2.5,\;p.\;2283]{kudla_2010_eisenstein}, the functions $\tilde{b}(m,s,T,\mu(\tilde{\varphi}_0))$ are holomorphic at $s = 1/2$ and therefore its residue vanishes for all $m\in\Z$
\end{proof}
\begin{cor}\label{corholo}
Let $\tilde{\varphi}_0\in\mathcal{S}(V(\A))$ be a $\prod_{p\nmid\infty}\mathrm{SL}_2(\Z_p)-$invariant function. We have
\[\int_{\mathcal{F}^T_2}f_{\mu_0}(\tau)\RESE(v,-1/2,\mu(\tilde{\varphi}_0))_0v^{\sigma}d\mu(\tau) = 0.\]
\end{cor}
\begin{proof}
By direct computation, we obtain 
\begin{align*}0 &= \int_{-1/2}^{1/2}\int_{\mathcal{F}^T_2}f_{\mu_0}(u+iv)\RESE(u+iv+r,-1/2,\mu(\tilde{\varphi}_0))v^{\sigma}\frac{dudv}{u^2}dr \\&= \int_{\mathcal{F}^T_2}f_{\mu_0}(u+iv)\int_{-1/2}^{1/2}\RESE(u+iv+r,-1/2,\mu(\tilde{\varphi}_0))drv^{\sigma}\frac{dudv}{u^2} \\&= \int_{\mathcal{F}^T_2}f_{\mu_0}(\tau)\RESE(v,-1/2,\mu(\tilde{\varphi}_0))_0v^{\sigma}d\mu(\tau)\end{align*}
where the first equality follows from the previous lemma, and the second equality is justified since the integration domains are compact.
\end{proof}
\bibliographystyle{plain}
\bibliography{main}
\end{document}